\documentclass[10pt,letterpaper]{amsart}
\usepackage{amssymb,graphicx}
\usepackage{enumerate}
\usepackage[T1]{fontenc} \usepackage{mathptmx} 
\newcommand{\indel}[1]{\partial/\partial #1}

\newtheorem{theorem}{Theorem}%[section]
\newtheorem{lemma}[theorem]{Lemma}
\newtheorem{proposition}[theorem]{Proposition}

\newtheorem{question}{Question}

\theoremstyle{remark}

\newtheorem{definition}[theorem]{Definition}

 \def\RR{{\mathbb R}}

\def\NN{{\mathbb N}} 

\newcommand{\lieg}{\mathfrak{g}}
\newcommand{\lieh}{\mathfrak{h}}
\newcommand{\liep}{\mathfrak{p}}

\newcommand{\liel}{\mathfrak{l}}

\newcommand{\Ad}{\mbox{Ad }}
\newcommand{\ad}{\mbox{ad }}

\title{Quasihomogeneous three-dimensional real analytic Lorentz metrics do not exist}

\author{Sorin Dumitrescu} \address{Universit\'e C\^ote d'Azur, Universit\'e Nice Sophia Antipolis, CNRS, Laboratoire J.-A. Dieudonn\'e, UMR 7351, Parc Valrose, 06108 Nice Cedex 2, France} \email{dumitres@unice.fr}
\author{Karin Melnick}\address{University of Maryland, College Park,  MD 20742,  USA} \email{karin@math.umd.edu}

\thanks{The authors acknowledge support from U.S. National Science Foundation grants DMS-1107452, 1107263, 1107367, "RNMS: Geometric Structures and Representation Varieties (the GEAR Network)."  Melnick was also supported during work on this project by a Centennial Fellowship from the American Mathematical Society and by NSF grants DMS-1007136 and 1255462.}

\begin{document}
\begin{abstract} We show that a germ  of  a real analytic Lorentz metric  on $\RR^3$  which is {\it locally homogeneous}  on an open set containing the origin in its closure is necessarily locally homogeneous.  We classifiy Lie algebras that can act \emph{quasihomogeneously}---meaning they act transitively on an open set admitting the origin in its closure, but  not at the origin---and isometrically for such a metric.  In the case that the isotropy at the origin of a quasihomogeneous action is semi-simple, we provide a complete set of normal forms of the metric and the action.  
\end{abstract}

\maketitle

\section{Introduction}

A Riemannian or pseudo-Riemannian metric is called {\it locally homogeneous} if any two points can be connected by flowing along a finite sequence of local Killing fields.  The study of such metrics is a traditional field in differential geometry.  In dimension two, they are exactly the semi-Riemannian  metrics of constant sectional curvature. Locally homogeneous  Riemannian metrics of dimension three are the subject of Thurston's 3-dimensional geometrization program~\cite{thurston}. The classification of compact locally homogeneous  Lorentz $3$-manifolds was given in~\cite{dumzeg}.

The most symmetric geometric structures after the locally homogeneous ones are those  which are~{\it quasihomogeneous}, meaning they are locally homogeneous on an open set containing the origin in its closure, but not locally homogeneous in the neighborhood of the origin. In particular, all the scalar invariants  of a quasihomogeneous  geometric structure are constant. Recall that, for  Riemannian metrics, constant scalar invariants implies local homogeneity (see~\cite{Tri} for an effective result).

 In a recent joint work with A. Guillot, the first author  obtained   the  classification of germs of quasihomogeneous, real analytic, torsion free, affine connections on surfaces~\cite{dumiguillot}.  The article~\cite{dumiguillot} also classifies the
 quasihomogeneous germs of real analytic, torsion free, affine connections which extend to {\it compact} surfaces. In particular, such germs  of  quasihomogeneous connections do exist.
 
 The first author proved in~\cite{Dumitrescu} that   {\it a real analytic Lorentz metric on a compact $3$-manifold which is locally homogeneous on a nontrivial open set is locally homogeneous on all of the manifold}.  In other words, quasihomogeneous real analytic Lorentz metrics do not extend  to {\it compact} threefolds.  The same is known to be true, by work of the second author, for real analytic Lorentz  metrics on compact manifolds of  higher dimension, under the assumptions that the Killing algebra is semisimple, the metric is geodesically complete, and the universal cover is acyclic~\cite{Melnick}.  In the smooth category, A. Zeghib proved in~\cite{Zeghib} that compact Lorentz $3$-manifolds which admit essential Killing fields are necessarily locally homogeneous.

Here we  simplify   arguments  of~\cite{Dumitrescu} and   introduce   new ideas  in order to dispense with the compactness assumption and  prove  the following local result:

\begin{theorem} \label{main} Let~$g$ be a   real-analytic Lorentz metric  in a  connected open   neighborhood  $U$ of the origin in~$\mathbf{R}^3$. If  $g$  is locally homogeneous on a nontrivial open subset in $U$, then
$g$ is locally homogeneous on all of $U$.
\end{theorem}

As a by-product of this new proof, we classifiy Lie algebras that can act isometrically for  a three-dimensional Lorentz metric and   \emph{quasihomogeneously}, meaning they act transitively on an open set admitting the origin in its closure, but  not at the origin.  In the case that the isotropy at the origin of such a quasihomogeneous action is semisimple, we provide a complete set of normal forms of the metric and the action, which, by Theorem \ref{main} above, are all locally homogeneous (see Proposition~\ref{prop.ss.classification} and Proposition~\ref{second}).

We also present a new approach to the problem in Section \ref{section5}, relying on the Cartan connection associated to a Lorentzian metric.  This approach yields a nice alternate proof of our results.

Our work is  motivated by Gromov's \emph{Open-Dense Orbit Theorem}~\cite{DG,Gro} (see also~\cite{Benoist2,Feres}). Gromov's result asserts that, if the pseudogroup of local automorphisms of  a \emph{rigid geometric structure}---such as a Lorentz metric or a connection---acts with a dense orbit, then this orbit is open.  In this case, the rigid geometric structure is locally homogeneous on an open dense set. Gromov's theorem says little about this maximal open and dense set of local homogeneity, which appears to by mysterious (see \cite[7.3.C]{DG}).  In many interesting geometric situations, it can be shown to be all of the connected manifold. This was proved, for instance,  for Anosov flows preserving a pseudo-Riemannian metric arising from differentiable stable and unstable foliations and a transverse contact  structure~\cite{BFL}. In~\cite{BF}, the authors deal with this question; their results indicate ways in which some rigid geometric structures cannot degenerate off the open dense set. 

The composition of this article is the following. In Section~\ref{section2} we use the geometry of Killing fields and geometric  invariant theory to prove that the Killing Lie algebra  of a  three-dimensional quasihomogeneous Lorentz metric $g$  is a three-dimensional, solvable, nonunimodular Lie algebra. We also show that $g$ is locally homogeneous away from a totally geodesic surface $S$, on which the isotropy is a one
parameter semisimple group or a one parameter unipotent group.  In the case of semisimple isotropy, Theorem~\ref{main} is proved in Section~\ref{section3}.  The proof of this case relies on the classification of normal forms of the metrics admitting quasihomogeneous isometric actions (see Proposition~\ref{prop.ss.classification} and Proposition~\ref{second}).
In the case of unipotent isotropy, Theorem~\ref{main} is proved in Section~\ref{section4}. Section~\ref{section5}  provides an  alternative proof of Theorem~\ref{main} using the formalism of Cartan connections.

Our result raises the following question:
\begin{question}
Let $g$ be a smooth Lorentz metric on a connected three-dimensional manifold $M$.  If $g$ is locally homogeneous on an open, dense subset of $M$, then must $g$ be locally homogeneous on all of $M$?
\end{question}

We are aware of noncompact quasihomogeneous examples of lower regularity $C^1$, recently discovered by C. Frances.  We would like to thank C. Frances for interesting conversations on the topic of this paper.  We thank the referee for her/his careful reading of our manuscript and many useful remarks.

\section{Killing Lie Algebra. Invariant Theory}   \label{section2}

Let  $g$ be a real analytic  Lorentz metric defined in a connected   open  neighborhood $U$ of the origin in $\RR^3$, which we assume is also simply connected.
In this section we recall the definition and several properties of the Killing algebra of $(U,g)$.  These were proved in \cite{Dumitrescu} without use of the compactness assumption.  For completeness, we briefly explain their derivation again here.

Classically,  (see, for instance~\cite{Gro,DG}) one  considers the $k$-jet of $g$ by taking at each point $u \in U$ the expression of $g$ up to order $k$ in exponential coordinates. In these coordinates, the $0$-jet of $g$ is
the standard flat Lorentz metric $dx^2 + dydz$. At each point $u \in U$, the space of exponential coordinates is acted on simply transitively by $O(2,1)$, the identity component of which is  isomorphic to $PSL(2, \RR)$.  
The space of all exponential coordinates in $U$ compatible with a fixed orientation and time orientation is a principal $PSL(2,\RR)$-bundle over $U$, which we will call the orthonormal frame bundle and denote by $R(U)$.

Geometrically, the $k$-jets of $g$ form an analytic $PSL(2,\RR)$-equivariant map $g^{(k)}: R(U) \to V^{(k)}$, where  $V^{(k)}$ is the  finite-dimensional vector space of $k$-jets at 0 of Lorentz metrics on $\RR^3$ with fixed 0-jet $dx^2+dydz$.
The group $O^0(2,1) \simeq PSL(2, \RR)$ acts linearly on this space, in which the origin corresponds to the $k$-jet of the flat metric.  
%For each $u \in U$, the group $PSL(2,\RR)$ acts on the $k$-jets at $u$ of the metric $g$, for all $k \in \NN$: choose a system of local exponential coordinates at $u$ and write the $k$-jet of $g$ in these coordinates.  Any linear automorphism of $(T_{u}U,g(u))$ gives another system of local exponential coordinates at $u$, in which the $k$-jet of $g$ is another point of $V^{(k)}$.  
One can find the details of this classical construction in~\cite{DG}.  
%A consequence of Gromov's Frobenius Theorem (\cite[Sec 3.4]{Gro}; see also \cite[Prop 3.10]{Melnick2}) is the existence, for each $x \in U$, of $k_0 \in \NN$ such that the stablizer in $\mathfrak{o}(2,1)$ of $g^{(k_0)}(b)$, for $b \in R(U)$ lying over $x$, corresponds to the Lie algebra of local Killing fields vanishing at $x$, in the frame given by $b$. 

Recall also that a local vector field is a {\it Killing field} for a Lorentz metric $g$ if its flow preserves $g$ wherever it is defined.  Note that local Killing fields preserve orientation and time orientation, so they act on $R(U)$.  The collection of all germs of local Killing fields at a point $u$ has the structure of a finite dimensional Lie algebra $\mathfrak{g}$ called the {\it local Killing algebra} of $g$ at $u$.  At a given point $u \in U$, the subalgebra $\mathfrak{i}$ of the local Killing algebra consisting of the local Killing fields $X$ with $X(u)=0$ is called the {\it isotropy} algebra at $u$.

The proof of Theorem~\ref{main} will use analyticity in an essential way.  We will make use of an extendability  result  for  local  Killing fields proved first by Nomizu in the real-analytic Riemannian setting~\cite{Nomizu} and generalized then for any $C^\omega$ rigid geometric structure by Amores and Gromov~\cite{Amores,Gro,DG}. This phenomenon states  that a local Killing field of $g$ can  be extended analytically along any curve $\gamma$  in $U$, and the resulting Killing field germ at the endpoint only depends on  the homotopy type of $\gamma$.  Because $U$ is assumed connected and simply connected, {\it local Killing fields extend to all of $U$.}  Therefore, the local Killing algebra at any $u \in U$ equals the algebra of globally defined Killing fields on $U$, which we will denote by $\lieg$.
 
\begin{definition} 
\label{def.lochom.transitive}
The  Lorentz metric  $g$ is \emph{locally homogeneous on an open subset $W \subset U$}, if for any $w \in W$ and any  tangent vector $V  \in T_{w}W$, there exists $X \in \lieg$ such that $X(w)=V.$  In this case, we will say that the Killing algebra $\mathfrak{g}$ is \emph{transitive on $W$}.
\end{definition}

Any two points in a  {\it connected}  open subset $W$ on which $g$ is locally homogeneous  can be related by flowing along a finite sequence of local Killing fields of $g$.

Notice  that Nomizu's extension phenomenon does not imply that the extension of a family of  pointwise linearly independent Killing fields stays linearly independent.  The assumption of Theorem 1 is that $\lieg$ is transitive on a nonempty open subset $W \subset U$.  Choose three elements $X,Y,Z \in \lieg$ that are linearly independent at a point $u_0 \in W$.  The function $\mbox{vol}_g(X(u),Y(u),Z(u))$ is analytic on $U$ and nonzero in a neighborhood of $u_0$.  The vanishing set of this function is a closed analytic proper  subset  $S'$ of $U$ containing  the points where $\lieg$ is not transitive.  Its complement  is an open dense set of $U$ on which $\lieg$ is transitive.  

From now on  we will assume that  {\it $g$ is a   quasihomogeneous Lorentz metric in the neighborhood $U$ of the origin in $\RR^3$}, with Killing algebra $\lieg$. Let $S$ be the complement of the maximal open subset of $U$ on which $\lieg$ acts transitively---that is, of a maximal locally homogeneous subset of $U$.  It is an intersection of closed, analytic  proper subsets, so \emph{$S$ is  a nontrivial closed and analytic subset of positive codimension passing through the origin}. The aim of this article is to prove that this is impossible.

%The  set of points $s$ in $U$ at which the Killing algebra  $\mathfrak{g}$ of $g$ does  not span the  tangent space $T_sU$  is a nontrivial locally closed  analytic subset $S$ in $U$ passing through 
%origin. In this case, $g$  is locally homogeneous on each connected component of $U \setminus S$, but not in the neighborhood of points in $S$.
 % Moreover, at  each point $s$ of $S$ one will get a {\it nontrivial isotropy algebra}  given by the kernel of the canonical evaluation morphism $ev : \mathfrak{g} \to T_sU$ (which is, by definition, of nonmaximal rank  at points of $S$).
 
We will next derive some basic properties of $\lieg$ that follow from quasihomogeneity.  
 
 \begin{lemma}[\cite{Dumitrescu} Lemme 3.2(i)]
\label{3}
The Killing algebra $\mathfrak{g}$ cannot be both  three-dimensional and unimodular.
 \end{lemma}
 
 \begin{proof} Let $(K_1$, $K_2$, $K_3)$ be a basis of the Killing algebra. Again consider the analytic function $v=vol_g(K_1, K_2, K_3)$. Since $\mathfrak{g}$ is unimodular and preserves the volume form of $g$, the function $v$ is nonzero and  constant on each open set where $\mathfrak{g}$ is transitive. On the other hand, $v$ vanishes  on $S$: a contradiction.
 \end{proof}

\begin{lemma}[\cite{Dumitrescu}  Lemme 2.1, Proposition 3.1, Lemme 3.2(i)]
 \label{dim}  

\ \ \

\begin{enumerate}[(i)]
\item The dimension of the isotropy  at a point $u \in U$ differs from two.

\item  The Killing algebra $\mathfrak{g}$  is of dimension three.
                             
\item The Killing algebra  $\mathfrak{g}$ is solvable.
\end{enumerate}
\end{lemma}

\begin{proof} (i) Assume for a contradiction that the isotropy algebra $\mathfrak{i}$ at a point $u \in U$ has dimension two. Elements of $\mathfrak{i}$ act linearly in exponential coordinates at $u$. Since elements of $\mathfrak{i}$  preserve $g$, they preserve, in particular,  the $k$-jet of $g$ at $u$, for all $k \in \NN$.
 This gives an embedding of $\mathfrak{i}$ in the Lie algebra of $PSL(2,\RR)$ such that the corresponding two-dimensional connected subgroup of $PSL(2,\RR)$ preserves the $k$-jet of $g$ at $u$, for all  $k \in \NN$.  But {\it stabilizers in a  finite-dimensional  linear  algebraic $PSL(2,\RR)$-action never have dimension two}.  Indeed, it suffices to check this statement for irreducible
 linear representations of $PSL(2, \RR)$, for which it is well-known that the stabilizer in $PSL(2,\RR)$  of a nonzero element is zero- or one-dimensional~\cite{Kir}.
 
 It follows that the stabilizer in $PSL(2,\RR)$ of the $k$-jet of $g$ at $u$ is of dimension three and hence equals $PSL(2, \RR)$.  Consequently, in exponential coordinates at $u$, each element of $\mathfrak{sl}(2,\RR)$ gives rise to a local linear vector field which preserves $g$, because it preserves all $k$-jets of the analytic metric $g$ at $u$.
%On the Lie algebra level, each element of $\mathfrak{sl}(2,\RR)$ preserves all $k$-jets of the analytic metric $g$ at $u$, and so, by the Frobenius Theorem, gives rise in exponential coordinates to a local Killing field of $g$.  
The isotropy algebra $\mathfrak{i}$ thus contains a copy of $\mathfrak{sl}(2, \RR)$: a contradiction, since $\mathfrak{i}$ was assumed of dimension two.
 
 (ii) Since $g$ is quasihomogeneous, the Killing algebra is of dimension at least $3$.  For a three-dimensional Lorentz metric, the maximal dimension of the Killing algebra is $6$. This characterizes Lorentz metrics of constant sectional curvature. Indeed, in this case,  the isotropy is, at each point,  of dimension three (see, for instance, ~\cite{Wolf}). These Lorentz metrics are locally homogeneous.
 
 Suppose that the Killing algebra of $g$  is of dimension $5$. Then, on any open set of local homogeneity the isotropy is two-dimensional. This is in contradiction with point (i).
 
 Last, suppose that the Killing algebra of $g$  is of dimension $4$. Then, at a point $s \in S$,  the isotropy has dimension $\geq 2$. Hence, point (i) implies that the isotropy  at  $s$  has dimension three and thus is isomorphic to $\mathfrak{sl}(2, \RR)$. Moreover, the  standard linear action of the isotropy on $T_sU$ preserves the image of the evaluation morphism $ev(s): \mathfrak{g} \to T_sU$, which is a line. But the standard $3$-dimensional  $PSL(2, \RR)$-representation does not admit invariant lines: a contradiction. 

Therefore, the Killing algebra is three-dimensional.
 
 (iii) A Lie algebra of dimension three is semisimple or solvable~\cite{Kir}. Since semisimple Lie algebras are unimodular, Lemma~\ref{3} implies that  $\mathfrak{g}$ is solvable.
 \end{proof}

Let us recall Singer's result~\cite{Singer,DG,Gro} which asserts that {\it $g$ is locally homogeneous if and only if the image of $g^{(k)}$ is exactly one $PSL(2, \RR)$-orbit in $V^{(k)}$, for a certain $k$ (big enough).}   This theorem is the key ingredient in the proof of the following fact.  

\begin{proposition}[\cite{Dumitrescu} Lemme 2.2]
 \label{invariant theory}  
If $g$ is quasihomogeneous, then the Killing algebra $\mathfrak{g}$ does not preserve any nontrivial vector field of constant norm $\leq 0$.
\end{proposition}

\begin{proof} Let $k \in \NN$ be given by Singer's Theorem.  First suppose, for a contradiction, that  there exists an isotropic vector field $X$ in $U$ preserved by $\mathfrak{g}$. Then the $\mathfrak{g}$-action  on $R(U)$, lifted from the action on $U$, preserves the subbundle $R'(U)$, where $R'(U)$ is a reduction 
of the structural group $PSL(2, \RR) \cong O^o(2,1)$ to the stabilizer of an isotropic vector in the standard linear representation on $\RR^3$:
$$H=
\left\{ 
\left(\begin{array}{cc} 1 & T \\ 0 & 1\\ \end{array}\right) \in PSL(2,\RR)
\ : \  T \in \RR
\right\}.
$$

Restricting to exponential coordinates with respect to frames the first vector of which is   $X$ gives an $H$-equivariant map $g^{(k)}:R'(U) \to V^{(k)}$.
On each open set $W$ on which  $g$ is locally homogeneous, the image  $g^{(k)}(R'(W))$ is exactly one $H$-orbit $\mathcal O \subset V^{(k)}$. Let $s \in S$ be a point in the closure of $W$. Then the image
under $g^{(k)}$ of the fiber $R'(W)_s$ lies in the closure of $\mathcal O$.  But $H$ is unipotent, and a classical result due to Kostant and Rosenlicht~\cite{Rosenlicht} asserts that 
{\it for algebraic representations of unipotent groups, the orbits are closed.} This implies that the image $g^{(k)}(R'(W)_s)$ is also $\mathcal O$. Moreover, this holds for all
$s \in S$. Indeed, the restriction of $\mathfrak{g}$ to  $S$ being  transitive  (as will be proved independently in point (i) of Lemma~\ref{unimodular}), this holds for all $s \in S$.

Any open set of local homogeneity in $U$ admits points of $S$ in its closure. It follows that the image of $R'(U)$ under $g^{(k)}$ is exactly the orbit $\mathcal O$.  Singer's  theorem
implies that $g$ is locally homogeneous, a contradiction to quasihomogeneity.

If there exists  a $\mathfrak{g}$-invariant    vector field $X$ in $U$ of constant strictly negative $g$-norm, then the  $\mathfrak{g}$-action on $R(U)$ preserves a subbundle $R'(U)$ with structural group $H'$, where
$H'$  is the stabilizer  of a strictly negative vector  in the standard linear representation of 
$PSL(2, \RR)$ on $\RR^3$.  In this case, $H'$  is a compact one parameter group in $PSL(2, \RR)$. The previous argument again yields a contradiction, after replacing the Kostant-Rosenlicht Theorem by the obvious  fact that  orbits of smooth compact group actions are closed.
\end{proof}

\begin{lemma}[compare \cite{Dumitrescu}, Proposition 3.3] 
\label{unimodular}   
After possibly shrinking $U$, we have

\begin{enumerate}[(i)]
%\ \ \
\item $S$ is a connected, real analytic submanifold of codimension one, on which $\lieg$ acts transitively.

\item  The isotropy at a point of $S$ is unipotent or $\RR$-semisimple.

\item The restriction of $g$ to $S$ is degenerate.
\end{enumerate}
\end{lemma}
  
\begin{proof}
(i) The fact that $S$ is a real analytic set was already established above: it coincides with the vanishing of the  analytic function $v=vol_g(K_1, K_2, K_3)$, where  $(K_1$, $K_2$, $K_3)$ is  a basis of the Killing algebra.  If needed, one can  shrink the open set $U$ in order that $S$ be connected.  By  point (i) in Lemma~\ref{dim}, the isotropy algebra at points in  $S$ has dimension one or three. We prove that this dimension must  be equal to one.

Assume, for a contradiction, that there exists $s \in S$ such that the isotropy at $s$ has dimension three. Then, the isotropy algebra at $s$ is isomorphic to $\mathfrak{sl}(2, \RR)$.
On the other hand, since both are $3$-dimensional,  the isotropy algebra at $s$ is isomorphic to $\mathfrak{g}$. Hence, $\mathfrak{g}$ is semisimple, which contradicts Lemma~\ref{dim} (iii).

It follows that the isotropy algebra at each point $s \in S$ is of dimension one.  Equivalently, the evaluation morphism $ev(s): \mathfrak{g} \to T_sU$ has rank two. Since the $\mathfrak{g}$-action preserves $S$, this implies that $S$ is a smooth submanifold of codimension one in $U$ and $T_sS$ coincides with the image of $ev(s)$.  The restriction of $\mathfrak{g}$ to $S$ satisfies Definition \ref{def.lochom.transitive}, so is transitive. 

(ii) Let $\mathfrak{i}$ be the isotropy Lie algebra at $s \in S$.  It corresponds to a 1-parameter subgroup of $PSL(2,\RR)$, which is elliptic, $\RR$-semisimple, or unipotent.  In any case, there is a tangent vector $V \in T_s U$ annihilated by $\mathfrak{i}$.  Then $\mathfrak{i}$ also vanishes along the curve $\exp_s(tV)$, where defined.  Because points of $U \backslash S$ have trivial isotropy, this curve must be contained in $S$.  Thus the fixed vector $V$ of the flow of $\mathfrak{i}$ is tangent to $S$.

If $\mathfrak{i}$ is elliptic, it preserves a tangent direction at $s$ transverse to the invariant subspace $T_sS \subset T_sU$.  Within $T_sS$, there must also be an invariant line independent from $V$.  But now an elliptic flow with three invariant lines must be trivial.  We conclude that $\mathfrak{i}$ is semisimple or unipotent.

(iii)  If the isotropy is unipotent, the vector $V$ annihilated by $\mathfrak{i}$ must be isotropic, and the invariant subspace $T_sS$ must equal $V^\perp$.  So $S$ is degenerate in this case.

If $\mathfrak{i}$ is semisimple over $\RR$, then $V$ is spacelike.  The other two eigenvectors of $\mathfrak{i}$ have nontrivial eigenvalues and must be isotropic.  On the other hand, $\mathfrak{i}$ preserves the plane $T_sS$, so it preserves a line of $T_sU$ transverse to $S$ and a line independent from $V$ in $T_sS$.  These lines must be the eigenspaces of $\mathfrak{i}$.  If the plane $T_sS \subset T_sU$ contains an isotropic line and is transverse to an isotropic line, then it is degenerate.  
\end{proof}

  According to Lemma~\ref{unimodular} we have two different geometric  situations, which will be treated separately in Sections~\ref{section3} and~\ref{section4}.  The case of $\RR$-semisimple isotropy will be referred to as just ``semisimple'' below.

  \section{No quasihomogeneous Lorentz metrics with semisimple isotropy}   \label{section3}
  
If the isotropy at $s \in S$ is semisimple, then it fixes a vector $V \in T_sS$ of positive $g$-norm.  Using the transitive $\mathfrak{g}$-action on $S$, we can extend $V$ to a $\mathfrak{g}$-invariant  vector field $X$ on $S$ with constant positive $g$-norm.  In this section {\it we assume that  the isotropy is semisimple}.  We can suppose  that $X$ is of constant norm equal to $1$.

  Recall  that the affine group of the real  line  $\operatorname{Aff}$  is the group of transformations of $\RR$ given by $x \mapsto ax+b$, with $a \in \RR^*$ and $b \in \RR$. If $Y$ is the infinitesimal generator of the one-parameter group of  homotheties and $H$ the infinitesimal generator of the one parameter group of  translations,  then $\lbrack Y, H \rbrack =H$.

\begin{lemma}[compare \cite{Dumitrescu}, Proposition 3.6]  
\label{Killing}  

\begin{enumerate}[(i)]

\item The Killing algebra  $\mathfrak{g}$ is isomorphic to $\RR \oplus \mathfrak{aff}$. The stabilizer of a point of $S$ corresponds to a one-parameter group of homotheties in 
$\operatorname{Aff}$. 

\item The vector field $X$ is the restriction to $S$ of a central element $X'$ in $\mathfrak{g}$.

\item The restriction of the Killing algebra to $S$ has, in adapted analytic coordinates $(x,h)$,  a basis $(-h \frac{\partial}{\partial h}, \frac{\partial}{\partial h},\frac{\partial}{\partial x})$.

\item  In the above coordinates, the restriction of $g$ to $S$ is $dx^2$.
\end{enumerate}
\end{lemma}

\begin{proof} 
(i)  We show  first that the derived Lie algebra $\lieg' = \lbrack \mathfrak{g}, \mathfrak{g}  \rbrack$ is 1-dimensional.  It is a general fact that the derived algebra of a solvable Lie algebra is nilpotent~\cite{Kir}. Remark first that $\lbrack \mathfrak{g}, \mathfrak{g} \rbrack \neq 0$. Indeed,  otherwise $\mathfrak{g}$ is abelian and the action of the isotropy $\mathfrak{i} \subset \mathfrak{g}$ at a point $s \in S$ is trivial on $\mathfrak{g}$ and hence on $T_{s}S$, which is identified with $\mathfrak{g} / \mathfrak{i}$. The isotropy action  on the tangent space $T_sS$ being trivial implies that the isotropy action is trivial on $T_{s} U$ (An element of $O(2,1)$ which acts trivially on a plane in $\RR^3$ is trivial). This implies that the isotropy is trivial at $s \in S$: a contradiction.  As $\mathfrak{g}$ is 3-dimensional, $\lieg'$ is a nilpotent Lie algebra of dimension $1$ or $2$, hence $\mathfrak{g}' \simeq \RR$, or $\mathfrak{g}' \simeq \RR^2$.

Assume, for a contradiction, that $\mathfrak{g}' \simeq \RR^2$.  We first  prove  that the isotropy $\mathfrak{i}$ lies in $\lbrack \mathfrak{g}, \mathfrak{g} \rbrack$.  Suppose this is not the case.  Then $\lbrack \mathfrak{g}, \mathfrak{g} \rbrack \simeq \RR^2$ acts 
  freely and transitively on $S$, preserving the vector field $X$.  Then $X$ is the restriction to $S$ of a Killing vector  field $X' \in \lbrack \mathfrak{g}, \mathfrak{g} \rbrack$.  

Let $Y$ be a generator of the isotropy at $s \in S$. Since $X$ is fixed by the isotropy, one gets, in restriction to $S$,  the following Lie bracket relation: $\lbrack Y, X' \rbrack = \lbrack Y,X \rbrack=aY$,
for some $a \in \RR$. On the other hand, by our assumption, $Y \notin \lbrack \mathfrak{g}, \mathfrak{g} \rbrack$, meaning that $a=0$.  This  implies that $X'$ is a central element in $\mathfrak{g}$. In particular, $\mathfrak{g}'$
is at most one-dimensional: a contradiction.  Hence  $\mathfrak{i} \subset \lbrack \mathfrak{g}, \mathfrak{g} \rbrack$. 
    
Now let $Y$ be a generator of $\mathfrak{i}$, $\{Y,X'  \}$ be generators of $\lbrack \mathfrak{g}, \mathfrak{g} \rbrack$, and   $(Y,X' ,Z)$      be a basis   of $\mathfrak{g}$. The tangent
           space of $S$ at a point $s  \in S$ is identified with $\mathfrak{g} /  \mathfrak{i}$. Denote $\bar{X'}, \bar{Z}$ the projections of $X'$ and $Z$ to this quotient.  The infinitesimal action of $Y$ on this tangent space is given      in the basis
           $\{ \bar{X'}, \bar{Z} \}$ by the matrix 
$$ad(Y) = \left(  \begin{array}{cc}
                                                                 0  &   *\\
                                                                 0     &  0\\
                                                                 
                                                                 \end{array}  \right) $$
because $\mathfrak{g}'  \simeq \RR^2$ and $ad(Y)(\mathfrak{g}) \subset \mathfrak{g}'.$  Moreover, $ad(Y) \neq 0$, since the restriction of the isotropy action to $T_sS$ is injective.  From this form of $ad(Y)$,  we see that the isotropy is unipotent with fixed direction $\RR X'$: a contradiction.

We have proved  that $\lbrack \mathfrak{g}, \mathfrak{g}  \rbrack$ is 1-dimensional. Notice that $\mathfrak{i} \neq \lbrack \mathfrak{g}, \mathfrak{g}  \rbrack$. Indeed, if they are equal, then the action of the isotropy on the tangent space $T_sU$ at $s \in S$ is trivial: a contradiction.  

Let $H$ be a generator of $\lbrack \mathfrak{g}, \mathfrak{g} \rbrack$, and $Y$ the generator of $\mathfrak{i}$.  Then $\lbrack Y, H \rbrack =aH$, with $a \in \RR$.  If $a=0$, then the image of $ad(Y)$, which  lies in $\lbrack \mathfrak{g}, \mathfrak{g} \rbrack$, belongs to   the kernel of $ad(Y)$,  which contradicts semisimplicity of the isotropy.  Therefore $a \neq 0$ and we can assume, by changing  the generator $Y$ of the isotropy, that $a=1$, so $\lbrack Y, H \rbrack =H$. 
  
  Let $X' \in \mathfrak{g}$ be such that  $\{X' ,H \}$ span the kernel of $ad(H)$. Then $( Y, X', H )$ is a basis of
$\mathfrak{g}$. There is $b \in \RR$ such that $\lbrack X', Y\rbrack =bH$.  After replacing 
 $X'$ by $X' +bH$,  we can assume  $\lbrack X', Y\rbrack =0$.  It follows that $\mathfrak{g}$ is the Lie algebra   $\RR \oplus \mathfrak{aff}(\mathbf{R})$. The Killing field $X'$ spans the center, the isotropy $Y$ spans the one-parameter group of homotheties, and $H$ spans the one-parameter group of translations.
 
\smallskip
 
(ii) This comes from the fact that  $X$ is the unique vector field tangent to $S$ invariant by $\lieg$.

\smallskip

(iii) The commuting  Killing vector fields $X'$ and $H$ are nonsingular  on $S$. This implies that, in adapted  coordinates $(x,h)$ on $S$, $H= \frac{\partial}{\partial h}$ and $X=\frac{\partial}{\partial x}$.  Because $[Y,X] = 0$, the restriction of $Y$ to $S$ has the expression $f(h)\frac{\partial}{\partial h}$, with $f$ an analytic function vanishing at the origin. The Lie bracket relation $\lbrack Y, H \rbrack =H$ reads

$$\left[ f(h) \frac{\partial}{\partial h},\frac{\partial}{\partial h} \right] = \frac{\partial}{\partial h},$$
and leads to $f(h)=-h$.  

\smallskip

(iv) Since  $H= \frac{\partial}{\partial h}$ and $X=\frac{\partial}{\partial x}$ are Killing fields, the restriction of $g$  to $S$ admits constant coefficients with respect to  the coordinates $(x,h).$ Since 
$H$ is expanded by the isotropy, it follows that $H$ is of constant $g$-norm equal to $0$. On the other hand, $X$ is of constant $g$-norm equal to one. It follows that the expression of  $g$ on $S$  is $dx^2$.
\end{proof}

\begin{lemma} 
\label{normal form} 
Assume $\lieg$ as in Lemma \ref{Killing} acts quasihomogeneously on $(U,g)$.  In adapted analytic  coordinates $(x,h,z)$ on $U$,
$$g=dx^2+ dhdz+Cz^2dh^2+Dzdxdh \qquad \mbox{for some} \ C, D \in \RR.$$
Moreover, in these coordinates, $\frac{\partial}{\partial x}$, $\frac{\partial}{\partial h}$, and $-h\frac{\partial}{\partial h}+z\frac{\partial}{\partial z}$ are Killing fields.
\end{lemma}

\begin{proof} 
%Being central,  $X'$ is of constant norm (equal to $1$ after normalization), which implies that $X'$ is geodesic. We consider the Killing field $X$ (of constant norm equal to $1$)  and the Killing field $H$ which commutes with $X$
%and which in restriction to $S$ is  nonsingular and of constant norm $0$ (which is  $\frac{\partial}{\partial h}$ on $S$). 
%We show now that $H$ is also geodesic in restriction to $S$. 
%In restriction to $S$, the vector field  $H$ is of constant norm equal to $0$. This implies that $\nabla_HH$ is orthogonal to $X$. It follows that $\nabla_HH= \lambda H$, for some $\lambda \in \RR$.
%Since the isotropy preserves $\nabla$ and expands $H$ this implies $\lambda=0$.
Consider the commuting Killing vector fields $X'$ and $H$ constructed  in  Lemma~\ref{Killing}. Their restrictions to $S$ have  the expressions $H= {\partial}/{\partial h}$ and $X={\partial}/{\partial x}$. Recall that on $S$, the vector field $H$ is of constant $g$-norm equal to $0$
and $X$ is of constant $g$-norm equal to one. Point (iv) in Lemma~\ref{Killing} also  shows that $g(X,H)=0$ on $S$. Moreover, being central, $X'$ is of constant $g$-norm   on $U \setminus S$, hence of constant $g$-norm one on  all of $U$.

Define a geodesic vector field $Z$ as follows. At each point $s \in S$, there exists a unique tangent vector $Z_s$, transverse to $T_sS$, such that $g(Z_s,Z_s)=0, \ g(X_s,Z_s)=0,$ and $g(H_s,Z_s)=1$. 
In fact, $Z_s$ spans the second isotropic line (other than that generated by $H_s$)   in  $X^{\bot}_s$.  In this line $Z_s$ is uniquely determined by the relation $g(H_s,Z_s)=1$.  Now $X'$ and $H$ are Killing and, in restriction to $S$, commute. So along $S$, the vector field $Z$ is stable by the flow of $X$ and $H$. Now extend $Z$ via the geodesic flow: 
$$Z(\exp_s(tZ_s)) := (\exp_s)_{*tZ_s}(Z_s) = \frac{\mbox{d}}{\mbox{d}t} \exp_s(tZ_s)$$
The resulting geodesic vector field is well defined on a sufficiently small open neighborhood of $S$ in $U$.  Since $X'$ and $H$ are Killing, their flows commute with the exponential map, so $Z$ commutes with $X'$ and $H$.

 The image of $S$ through the flow of  $Z$ defines a foliation by surfaces. Each leaf is given by $\exp_S(zZ)$, for some $z$ small enough. The leaf $S$ corresponds to $z=0$.

 Let $(x,h,z)$ be analytic coordinates in the neighborhood of the origin such that $X'={\partial}/{\partial x}, H={\partial}/{\partial h}, Z={\partial}/{\partial z}$.
The scalar product $g(Z,X')$ is constant along the orbits of $Z$. This comes from the following classical computation~:
$$Z \cdot g(X',Z)=g(\nabla_Z{X'}, Z)+g(X',\nabla_ZZ)=0$$ 
since $\nabla_ZZ=0$ and $\nabla_{\cdot}X'$ is skew-symmetric with respect to $g$.  The same is true for $g(Z, H)$. In particular, the coefficients in $g$ of $dxdz$ and $dhdz$ are constant on the orbits of $Z$.

Moreover, the invariance of the metric by the commutative Killing algebra generated by
$X'$ and $H$ implies that $dxdz$ and $dhdz$ are also constant along the orbits of $X'$ and of $H$. This implies that the coefficients of $dxdz$ and $dhdz$ are $0$ and $1$, respectively, not only on $S$, but   over all of $U$.  

The coefficients of $dh^2$ and $dxdh$ depend only on $z$.  Then 
$$g=dx^2+ dhdz+c(z)dh^2+d(z)dxdh$$ 
with $c$ and $d$ analytic functions which both vanish at $z=0$.

Next we use the invariance of $g$  by $Y$.  Recall that $\lbrack Y,X' \rbrack =0$ and $\lbrack Y,H \rbrack=H$. Note that $Y$ preserves the two isotropic directions of $X'^{\bot}$, which are spanned by $Z$ and $H-d(z)X'$. From $g(X',H-d(z)X') \equiv 1$, compute
\begin{eqnarray*}
0 & = & Y.(g(X',H-d(z)X')) = g([Y,X'],H-d(z)X') + g(X',[Y,H-d(z)X']) \\
& = & g(X',H) - (Y.d) g(X',X') = d(z) - (Y.d)(z),
\end{eqnarray*}
so $Y.d = d$.  Then $[Y,H-d(z)X'] = H-d(z)X'$.
Next, from $g(H-d(z)X',Z) \equiv 1$, 
$$ 0 = g([Y,H-d(z)X'],Z) + g(H-d(z)X',[Y,Z])= 1 + g(H-d(z)X',[Y,Z]),$$ 
so $\lbrack Y, Z \rbrack =-Z$.
Now, since $Y$ and $X'$ commute, the general expression for $Y$ is 
$$Y = u(h,z) \frac{\partial}{\partial h} + v(h,z) \frac{\partial}{\partial z}+t(h,z)\frac{\partial}{\partial x}$$
with $u,v$, and $t$ analytic functions, where $u(h,0) = -h$, and $v$ and $t$ vanish on $\{ z=0 \}$.

The other Lie bracket relations read 
$$ \lbrack u(h,z) \frac{\partial}{\partial h} + v(h,z) \frac{\partial}{\partial z}+t(h,z)\frac{\partial}{\partial x}, \frac{\partial}{\partial h} \rbrack =\frac{\partial}{\partial h} $$ 
and
$$ \lbrack u(h,z) \frac{\partial}{\partial h} + v(h,z) \frac{\partial}{\partial z}+t(h,z)\frac{\partial}{\partial x}, \frac{\partial}{\partial z} \rbrack =-\frac{\partial}{\partial z}. $$
 The first relation gives
 $$ \frac{\partial u}{\partial h}=-1 \qquad  \frac{\partial v}{\partial h}=0 \qquad \frac{\partial t}{\partial h}=0.$$
  The second one leads to  
$$\frac{\partial u}{\partial z}=0 \qquad \frac{\partial v}{\partial z}=1 \qquad \frac{\partial t}{\partial z}=0.$$
  We get 
$$u(h,z)=-h \qquad v(h,z)=z \qquad t(h,z)=0.$$
  Hence, in our  coordinates,
$Y=-h {\partial}/{\partial h}+z \partial/\partial z$. The invariance of $g$ under the action of this linear vector field implies $c(e^{-t}z) e^{2t}=c(z)$ and
$d(e^{-t}z)e^t=d(z)$, for all $t \in \RR$. This implies then that $c(z)=Cz^2$ and $d(z) = Dz$, with $C, D$ real constants.
\end{proof}

\subsection{Computation of the Killing algebra}

We need to understand now whether  the  metrics $$g_{C,D}=dx^2+ dhdz+Cz^2dh^2+Dzdxdh$$ constructed in Lemma~\ref{normal form} really are quasihomogeneous. In other words, do the metrics in this family
 admit other Killing fields than ${\partial}/{\partial x}$, ${\partial}/{\partial h}$ and $- h {\partial}/{\partial h} +z {\partial}/{\partial z}$ ?  In this section we compute the full Killing algebra  $\lieg$ of $g_{C,D}$.
In particular, we obtain that the metrics $g_{C,D}=dx^2+ dhdz+Cz^2dh^2+Dzdxdh$ always admit additional Killing fields and, by Lemma \ref{dim} (ii)  are locally homogeneous.
  
The formula for the Lie derivative of $g$ (see, eg, ~\cite{kobayashi-nomizu}) gives
  $$( L_{T}g_{C,D}) \left( \frac{\partial}{\partial x_i}, \frac{\partial}{\partial x_j} \right)=T \cdot g_{C,D} \left( \frac{\partial}{\partial x_i}, \frac{\partial}{\partial x_j} \right) + g_{C,D}\left( \left[ \frac{\partial}{\partial x_i}, T \right], \frac{ \partial}{\partial x_j} \right)+g_{C,D} \left( \frac{\partial}{\partial x_i}, \left[ \frac{\partial}{\partial x_j}, T \right] \right).$$
   
Let $T=\alpha \indel{x} + \beta \indel{h} + \gamma \indel{z}$.  
The pairs 
$$\left( \frac{\partial}{\partial x_i}, \frac{\partial}{\partial x_j} \right) =
(1) \left( \frac{\partial}{\partial z}, \frac{\partial}{\partial z} \right)  \ \   (2) \left(  \frac{\partial}{\partial x}, \frac{\partial}{\partial x} \right) \ \  (3) \left(  \frac{\partial}{\partial x}, \frac{\partial}{\partial z} \right) \ \
(4)  \left( \frac{\partial}{\partial x}, \frac{\partial}{\partial h} \right) \ \ (5) \left(  \frac{\partial}{\partial h}, \frac{\partial}{\partial z} \right) \ \ (6) \left(  \frac{\partial}{\partial h}, \frac{\partial}{\partial h} \right) 
$$

give the following system of PDEs on $\alpha, \beta$ and $\gamma$ in order for $T$ to be a Killing field:
   
\begin{eqnarray}
\label{primera} 0 & = & \beta_{z},\\
\label{segunda} 0 & = & \alpha_{x}+Dz  \beta_{x},\\
\label{tercera} 0 & = & \beta_{x}+Dz \beta_z+ \alpha_z,\\
\label{cuarta} 0 & = & \gamma D+Dz \alpha_x+Cz^2 \beta_x+ \gamma_x+ \alpha_h +Dz \beta_h,\\
\label{quinta} 0 & = & \beta_{h}+ Cz^2 \beta_z+Dz \alpha_z+ \gamma_z,\\
\label{sexta} 0 & = & zC \gamma +Cz^2 \beta_h + Dz \alpha_h+ \gamma_h.
\end{eqnarray}

The following proposition finishes the proof of Theorem~\ref{main} in the case of semisimple isotropy on $S$:

\begin{proposition} \label{loc hom}
The Lorentz metrics $g_{C,D}$ are locally homogeneous for all $C,D \in \RR$.
\end{proposition}

\begin{proof}  
%If $C=0$,  one directly checks that $\alpha =\beta =0$ and $\gamma=e^{-Dx}$ is a solution of the PDE system, meaning that $e^{-Dx} \indel{z}$ is an extra Killing field  Killing field.  
It is straightforward to verify that 
$$T= Dh \frac{\partial}{\partial x} + \frac{1}{2}(D^2 - C)h^2  \frac{\partial}{\partial h} + ((C-D^2) zh- 1)  \frac{\partial}{\partial z}$$
satisfies equations (1)--(6).  Note that $T(0) = - \partial/\partial z$, so $T \notin \lieg$, and $(U,g)$ is locally homogeneous.
%Of course for $C=D=0$, the metric $g_{C,D}$ is the standard flat metric.
%If  $D \neq 0$, one directly checks that an extra Killing field is given by~:
% $T=ah \indel{x}+ \frac{1}{2}bh^2 \indel{h} +(-bzh- \frac{a}{D}) \indel{z}$, with $a,b$ such that $a(D - \frac{C}{D})=b$.
\end{proof}

We explain now our method to find the extra Killing field $T$  in Proposition~\ref{loc hom}, and we compute the full Killing algebra, $\lieg$, of $g_{C,D}$.  Recall the $n$-dimensional Lorentzian manifolds $\mbox{AdS}^n, \mbox{Min}^n,$  and $\mbox{dS}^n$, of constant sectional curvature $-1, 0$, and $1$, respectively (see, eg, \cite{Wolf}).  Recall also that $\mbox{AdS}^3$  is isometric to $SL(2,\RR)$ with the bi-invariant Cartan-Killing metric.

\begin{proposition} 
\label{prop.ss.classification}
\ \ 
\begin{enumerate}[(i)]
\item If $D \neq 0$ and $C \notin \{0, D^2\}$, then $(U,g_{C,D})$ is locally isometric to a left-invariant metric on $SL(2, \RR)$  with $\lieg \cong \RR \oplus \mathfrak{sl}(2, \RR)$. The isotropy is the graph of a Lie algebra homomorphism of the $\RR$ factor to the subalgebra spanned by a $\RR$-semisimple element of $\mathfrak{sl}(2,\RR)$.

\item If $D \neq 0$ and $C = D^2$, then $(U,g_{C,D})$ is locally isometric to a left-invariant metric on the Heisenberg group with $\lieg \cong \RR \ltimes \mathfrak{heis}$.  The isotropy is the $\RR$ factor, which acts by a semisimple automorphism of $\mathfrak{heis}$.

\item If $C=0$ and $D \neq 0$, then $(U,g_{C,D})$  is locally isometric to $\mbox{AdS}^3$, so $\lieg \cong \mathfrak{sl}(2,\RR) \oplus \mathfrak{sl}(2,\RR)$.

\item  If  $C \neq 0$ and $D =0$, then $(U, g_{C,D})$ is locally isometric to $\RR \times \mbox{dS}^2$, for which $\lieg \cong \RR \oplus \mathfrak{sl}(2, \RR)$.  The isotropy is generated by a semisimple element of $\mathfrak{sl}(2,\RR)$. 

\item  If $C=0$ and $D=0$, then $(U,g_{C,D})$ is locally isometric to $\mbox{Min}^3$, so $\lieg \cong \mathfrak{sl}(2,\RR) \ltimes \RR^3$.
\end{enumerate}
\end{proposition}

\begin{proof}  Recall that  $(x,h,z)$ are  analytic  coordinates on $U$, with $S=z^{-1}(0)$, such  that all 
 Lorentz metrics $g_{C,D}$ admit the Killing fields $X'=\frac{\partial}{\partial x}$, $Y=-h\frac{\partial}{\partial h}+z\frac{\partial}{\partial z}$ and $H=\frac{\partial}{\partial h}$, for which the Lie bracket relations are  $\lbrack Y,X' \rbrack = \lbrack H, X' \rbrack =0$ and $\lbrack Y,H \rbrack=H$.
Moreover, Proposition~\ref{loc hom} shows that all Lorentz metrics  $g_{C,D}$ are locally homogeneous and that  their  full Killing algebra $\mathfrak{g}$  is of dimension at least four.  In particular, the  Killing algebra $\mathfrak{g}$ strictly contains the previous   three-dimensional Lie algebra as a subalgebra $\liel$ acting quasihomogeneously in the neighborhood of the origin.

Assuming $g_{C,D}$ is not of constant sectional curvature, then Lemma \ref{dim} (i) implies $\mbox{dim } \lieg = 4$.  We first derive some information on the algebraic structure of $\lieg$ in this case.     

If $\mbox{dim } \lieg = 4$, then it is generated by $X', Y, H,$ and an additional Killing field $T$.  Since  the isotropy $\RR Y$ at the origin  fixes the spacelike vector $X(0)$ and expands $H$, we can choose
             a fourth generator  $T$ of $\mathfrak{g}$ evaluating at the origin to a generator  of the second isotropic direction of the Lorentz plane $X(0)^{\bot}$.  As the action of $Ad(Y)$ on $\mathfrak{g}$ is $g$-skew symmetric, we get at the origin : $\lbrack Y, T \rbrack (0)=-T(0)$. Hence $\lbrack Y, T \rbrack =-T + aY$ for some
             constant $a \in \RR,$ and we can replace $T$ with  $T -aY$ in order that $\lbrack Y,T \rbrack =-T$.  Since $X'$ and $Y$ commute, $[X',T]$   is also  an eigenvector of $ad(Y)$ with eigenvalue $-1$.  This eigenspace of $ad(Y)$ is one-dimensional, so $\lbrack T, X' \rbrack= c T$, for some  $c \in \RR$.
         
The Jacobi relation 
$$\lbrack Y, \lbrack T, H \rbrack \rbrack=  \lbrack  \lbrack Y, T \rbrack , H \rbrack +  \lbrack T, \lbrack Y, H \rbrack \rbrack =
     \lbrack -T, H \rbrack + \lbrack T, H \rbrack =0$$
says that $ \lbrack T, H \rbrack $ commutes with  $Y$.   The centralizer  of $Y$ in $\lieg$ is $\RR Y \oplus \RR X'$.  We conclude that $\lbrack H,T \rbrack =aX'- bY$, for some  $a,b \in \RR$.    

(i)  Assume $D \neq 0$ and $C \notin \{0, D^2\}$. A straightforward computation shows that $g_{C,D}$ is not of constant sectional curvature.   We will construct  a Killing field $T = \alpha \indel{x} + \beta \indel {h} + \gamma \indel{z}$, meaning the functions $\alpha$, $\beta$ and $\gamma$ solve the PDE system (1)--(6). We will moreover construct it so that $c=0$ and $a =1$.

First we use  the Lie bracket relations derived above for $T$ and $\liel$.
  Remark  that,  since $T$ and $X'$ commute, the coefficients $\alpha, \beta$ and $\gamma$ of $T$ do not depend on the coordinate $x$; in particular, equation (\ref{segunda}) is satisfied.
  The relation $\lbrack H, T\rbrack =aX' -bY$ reads, when $a = 1$, 
$$\left[  \frac{\partial}{\partial h}, T \right] = \frac{\partial}{\partial x} + b \left( h  \frac{\partial}{\partial h} - z  \frac{\partial}{\partial z} \right).$$
  This leads to $\alpha_h=1, \beta_h=bh,$ and $\gamma_h=-bz$.
  Using equation (\ref{primera}), we obtain $\beta=\frac{1}{2}bh^2 + \beta_0$.  We can take the additive constant $\beta_0 = 0$ because $\frac{\partial}{\partial h} \in \mathfrak{l}$.  Now equation (\ref{cuarta}) gives $\gamma=-bz h - 1/D$.
  
Equation (\ref{sexta}) now reads
$$ 0 = zC( - \frac{1}{D} - zbh) + Cz^2bh + Dz - bz = - \frac{Cz}{D} + Dz - bz$$
which yields $b = D - C/D$.  Now $\gamma$ can be written $-1/D - zh(D - C/D)$.
  
Equation (\ref{tercera}) says $\alpha_z = 0$, so we conclude $\alpha = h$.
The resulting vector field is 
\begin{eqnarray}
\label{eqn.T}
 T = h  \frac{\partial}{\partial x} + \frac{1}{2}(D - \frac{C}{D}) h^2  \frac{\partial}{\partial h} +  \left( zh ( \frac{C}{D} - D)- \frac{1}{D} \right) \frac{\partial}{\partial z}.
\end{eqnarray}
Note that the coefficients of $T$ also satisfy equation (\ref{quinta}), so $T$ is indeed a Killing field.

We obtained this solution setting $c=0$, so the Lie algebra $\lieg$ generated by $\{ T,X',Y,H \}$ contains $X'$ as a central element.  We also set $a=1$, and found $b = D - C/D$, so $[H,T]= X' + (C/D - D) Y$, which we will call $Y'$.  It is straightforward to verify that for $T$ as above, $[Y,T] = -T$.  Under the hypothesis $C \neq D^2$, the Lie subalgebra generated by $\{ Y', H, T \}$ is isomorphic to $\mathfrak{sl}(2,\RR)$, with $Y'$ $\RR$-semisimple, and it acts transitively on $U$.    Consequently, $g_{C,D}$ is locally isomorphic to a left-invariant Lorentz  metric on $SL(2, \RR)$.  The full Killing algebra is  $\lieg \cong \RR \oplus \mathfrak{sl}(2, \RR)$, with center generated by $X'$, and isotropy $\RR Y = \RR(X' + Y')$.  This terminates the proof of point (i).

\smallskip

(ii) When $D \neq 0$ and $C= D^2$, then (\ref{eqn.T}) still solves the Killing equations.  The bracket relations are the same, but now $[H,T] = X'$.  Then $\lieg \cong \RR \ltimes \mathfrak{heis}$, where the $\mathfrak{heis}$ factor is generated by $\{ H, T, X' \}$  and  acts transitively, and the  $\RR$ factor is generated by the isotropy $Y$, which acts by a semisimple automorphism on $\mathfrak{heis}$.  Up to homothety, there is a unique left-invariant Lorentz metric on $\mbox{Heis}$ in which $X'$ is spacelike, by Proposition 1.1 of \cite{dumzeg}, where it is called the \emph{Lorentz-Heisenberg geometry}.   

\smallskip

 (iii)  When  $C=0$ and $D \neq 0$, then (\ref{eqn.T}) again solves the Killing equations.  It now simplifies to
$$  T = h  \frac{\partial}{\partial x} + \frac{1}{2}D h^2  \frac{\partial}{\partial h} +  \left( -zh D- \frac{1}{D} \right) \frac{\partial}{\partial z}.$$

The bracket relation is $[H,T] = X' - DY$, and $\lieg$ still contains a copy of $\RR \oplus \mathfrak{sl}(2, \RR)$, with center generated by $X'$ and $\mathfrak{sl}(2, \RR)$ generated by $\{ X'- DY,H,T' \}$.  The $\mathfrak{sl}(2, \RR)$ factor still acts simply transitively. 
On the other hand, one directly checks that $\alpha =\beta =0$ and $\gamma=e^{-Dx}$ is a solution of the PDE system, meaning that $e^{-Dx} \indel{z}$ is also a Killing field.  From 
$$ \left[ X', e^{-Dx} \frac{\partial}{\partial z} \right] = - D e^{-Dx}  \frac{\partial}{\partial z} \neq 0$$
 it is clear that this additional Killing field does not belong to the subalgebra generated by $\{ T, X', Y, H \}$, in which $X'$ is central.   It follows that the Killing algebra is of dimension at least five, hence six by Lemma~\ref{dim} (i), which implies that $g_{0,D}$ is of constant sectional curvature.  Since $g_{0,D}$ is locally isomorphic to a left-invariant Lorentz metric
 on $SL(2, \RR)$, the sectional curvature is negative. Up to normalization, $g_{0,D}$ is locally isometric to $\mbox{AdS}^3$.
 
\smallskip

 (iv)  The Killing field $T$ in (\ref{eqn.T}) multiplied by $D$ gives
$$ T_D = D h  \frac{\partial}{\partial x} + \frac{1}{2}(D^2 - C) h^2  \frac{\partial}{\partial h} +  \left( zh ( C - D^2)- 1 \right) \frac{\partial}{\partial z}.$$ 

Setting $C \neq 0$ and $D =0$ gives
$$T_0 =   - \frac{Ch^2}{2}   \frac{\partial}{\partial h} +  \left( zhC - 1 \right) \frac{\partial}{\partial z}$$ 
which is indeed  a Killing field of $g_{C,0}$.  The brackets are $[X',T_0] = 0, [H,T_0] = CY$, and $[Y,T_0] = - T_0$.  As in case (i), the Killing  Lie algebra   contains a copy of $\RR \oplus \mathfrak{sl}(2, \RR)$, with center generated by $X'$, and $\mathfrak{sl}(2, \RR)$ generated by $\{ Y,H,T_0 \}$.
Here the isotropy generator $Y$ lies in the $\mathfrak{sl}(2,\RR)$-factor, which acts with two-dimensional orbits.  This local $\mathfrak{sl}(2,\RR)$-action defines a two-dimensional foliation  tangent to $X'^{\bot}$.  Recall that $X'$ is of constant $g$-norm equal to one, so $X'^{\bot}$ has Lorentzian signature.  The metric is, up to homotheties on the two factors,  locally isomorphic to the product $\RR \times \mbox{dS}^2$.

\smallskip

(v) If   $C=D=0$, then $g_{C,D}$ is flat and $\lieg \cong \mathfrak{sl}(2,\RR) \ltimes \RR^3$.
\end{proof}

As a by-product of the proof  of Theorem~\ref{main} in the case of semisimple isotropy, we have obtained  the following more technical result:

\begin{proposition} \label{second}  Let~$g$ be a   real-analytic Lorentz metric  in a neighborhood of the origin in $\RR^3$. Suppose  that there exists a three-dimensional   subalgebra $\liel$  of the Killing Lie algebra acting transitively on an
open set admitting the origin in its closure, but not in the neighborhood of the origin.  If the isotropy at the origin is a one-parameter 
$\RR$-semisimple subgroup in $O(2,1)$, then 

\begin{enumerate}[(i)]
\item There exist local analytic coordinates $(x,h,z)$ in the neighborhood of the origin and real constants $C,D$ such that  
$$g = g_{C,D}= dx^2+ dhdz+Cz^2dh^2+Dzdxdh.$$

\item The algebra $\liel$ is solvable, and equals, in these coordinates, 
$$\mathfrak{l}=\langle \frac{\partial}{\partial x}, \frac{\partial}{\partial h}, - h \frac{\partial}{\partial h} + z  \frac{\partial}{\partial z} \rangle.$$
In particular, $\liel \cong \RR \oplus \mathfrak{aff}(\RR)$,  where $\mathfrak{aff}(\RR)$ is  the Lie algebra of the affine group of the real line. 

\item All the metrics $g_{C,D}$ are locally homogeneous.  They admit a Killing field $T \notin \liel$ of the form
$$ T = Dh \frac{\partial}{\partial x} + \frac{1}{2}(D^2 - C)h^2  \frac{\partial}{\partial h} + ((C-D^2) zh- 1)  \frac{\partial}{\partial z}.$$
The possible geometries on $(U,g_{C,D})$ are given by (i) - (v) of Proposition \ref{prop.ss.classification}. 
\end{enumerate}
\end{proposition}

 \section{No quasihomogeneous Lorentz metrics with unipotent  isotropy}   \label{section4}
 
 We next treat the unipotent case of Lemma~\ref{unimodular}.  The following results can be found in~\cite{Dumitrescu} Propositions 3.4 and 3.5 in Section 3.1, where they are proved without making use of compactness.  See also \cite[Proposition 9.2]{Zeghib} for point (iii).  
  
 \begin{proposition} \label{surface}
\ \ 
\begin{enumerate}[(i)]
\item The surface $S$ is totally geodesic.
 
\item The Levi-Civita connection   $\nabla$ restricted to $S$ is either flat, or locally isomorphic to the canonical bi-invariant connection on the  affine group of the real line $\operatorname{Aff}$.
 
\item The restriction of the  Killing algebra $\mathfrak{g}$  to $S$  is isomorphic either to the Lie algebra of the  Heisenberg group in the flat case, or otherwise to a solvable  subalgebra $\mathfrak{sol}(1, a)$ of  $\operatorname{Aff} \times \operatorname{Aff}$, spanned  by the elements
 $(t,0), (0,t)$ and $(w, a w)$, where $t$ is the infinitesimal generator of the one-parameter group of translations, $w$ the infinitesimal generator of the one-parameter group of homotheties, and
 $a \in \RR$.
\end{enumerate}
 \end{proposition}
 
Recall that, as $S$ has codimension one,  the restriction to $S$  of the  Killing Lie algebra $\mathfrak{g}$ of $g$ is an isomorphism.   The Heisenberg group and $\mathfrak{sol}(1, -1)$ are  unimodular, so by Lemma~\ref{3}, $\mathfrak{g}$ is isomorphic to $\mathfrak{sol}(1, a)$, with $a \neq -1,$ and {\it $S$ is non flat}.
 
 Recall that in dimension three, the curvature is completely determined by its Ricci tensor, which is a symmetric bilinear form.  The Ricci tensor is determined by the Ricci operator,
 which is a   field of $g$-symmetric endomorphisms $A : TU \to TU$ such that $\operatorname{Ricci}(u,v)=g(Au,v)$, for any tangent vectors $u,v$.

\begin{definition} The metric $g$ is said to be {\it curvature homogeneous} if for any pair of points $u,u' \in U$, there exists a linear isomorphism from $T_uU$ to $T_{u'}U$ preserving both $g$ and the curvature tensor.
\end{definition}

 In dimension three, it is equivalent  to assume in the previous definition  that these linear maps  preserve both $g$ and  the Ricci operator $A$.
 
 \begin{proposition}    
\label{eigenvalues} 
\ \ 
\begin{enumerate}[(i)]
 \item The only eigenvalue of the Ricci operator is $0$, everywhere on $U$.
 
\item The metric $g$ is curvature homogeneous; more precisely, in an  adapted framing on $U$, the Ricci operator  reads 
$$ A = \left(  \begin{array}{ccc}
                                                                 0    &    0   & \alpha \\
                                                                 0     &  0  &    0     \\
                                                                 0     &   0  &  0 \\
                                                                 \end{array} \right), \qquad \alpha \in \RR^*.
$$
\end{enumerate}                                                                
\end{proposition}

\begin{proof}  (i)  Pick a point $s$ in $S$. The Ricci operator $A(s)$ must be invariant by the unipotent isotropy (which identifies with the stabilizer in the orthogonal group of $g(s)$ of an isotropic vector $X(s) \in T_sU$).

 The action of the isotropy  on  $T_{s} U$  fixes an isotropic vector $e_{1}=X(s)$ tangent to $S$  and so preserves the degenerate plane $e_{1}^{\bot} = T_sS$. 
  In order to define an adapted basis,   consider  two vectors $e_{2}, e_{3} \in T_{s} U$ such that
    $$g(e_{1}, e_{2})=0 \qquad g(e_{2},e_{2})=1 \qquad g(e_{3},e_{3})=0 \qquad g(e_{2},e_{3})=0 \qquad g(e_{3}, e_{1})=1$$ 
 
%  Note that such an adapted basis is uniquely determined by the choice of an unitary vector $e_{2} \in e_{1}^{\bot}$. Indeed, then $e_{3}$ is   uniquely defined in 
% $e_{2}^{\bot}$ by the relation  $g(e_{3}, e_{1})=1$ ( $e_{1}$ and  $e_{3}$ generate the two isotropic directions in $e_{2}^{\bot}$). 
% 
 The action on $T_{s} U$ of the one-parameter group of  isotropy is given in the basis $(e_{1}, e_{2}, e_{3})$ by  the matrix 
 $$ L_t= \left(  \begin{array}{ccc}
                                                                 1   &   t & - \frac{t^2}{2}\\
                                                                 0     &  1 &  -t\\
                                                                 0     &   0  &  1\\
                                                                 \end{array}  \right), \qquad t \in \RR.
$$

        First we show that   $A(s) : T_sU \to T_sU$   has,  in our adapted basis,  the following  form:                                                               
$$  \left(  \begin{array}{ccc}
                                                                 \lambda    &   \beta  & \alpha \\
                                                                 0     &  \lambda  &  - \beta  \\
                                                                 0     &   0  &  \lambda \\
                                                                 \end{array} \right), \qquad \alpha, \beta, \lambda \in \RR.$$
      Since $A(s)$ is invariant by the isotropy, it commutes with $L_{t}$ for all $t$. 
      Each eigenspace of $A(s)$ is  preserved by $L_{t}$, and eigenspaces of $L_t$ are preserved by $A(s)$.  As $L_{t}$ does not  preserve any non trivial splitting 
      of $T_sU$, it follows that all eigenvalues of $A(s)$ are equal to some $\lambda \in \RR$.  Moreover, the unique line and plane invariant by $L_t$ must also be invariant by $A(s)$, so $A(s)$ is upper-triangular in the basis $(e_1, e_2, e_3)$.  A straightforward calculation of the top corner entry of $A(s) L_t = L_t A(s)$ leads to the relation on the $\beta$ entries and thus to our claimed form for $A(s)$.

      Now the $g$-symmetry of $A(s)$ means $g(A(s)e_2,e_3)=g(e_2, A(s)e_3)$, which gives $\beta =0$.  Since the symmetric functions of the eigenvalues of $A$ are  scalar invariants, they must be constant on all of $U$. This implies that the only eigenvalue of $A$ is $\lambda$, on all of $U$.  It remains only to prove that $\lambda =0$. Consider an open set in $U$ on which the Killing algebra $\mathfrak{sol}(1, a)$ is transitive, so $g$ is locally isomorphic to a left-invariant  Lorentz metric
      on $SOL(1, a)$.  

The sectional and Ricci curvatures and Ricci operator of a left-invariant Lorentz metric on a given Lie group can be calculated, starting from the Koszul formula, in terms of the brackets between left-invariant vector fields forming an adapted framing of the metric.  In~\cite{CK} Calvaruso and Kowalski calculate Ricci operators for left-invariant  Lorentz metrics on three-dimensional Lie groups, assuming they are not symmetric (see also previous curvature calculations in \cite{nomizu_lorentz}, \cite{cordero_parker}, \cite{calv_einstein}).  If the metric on $U \backslash S$ were symmetric, then the covariant derivative of the curvature would vanish on all of $U$, which would imply $U$ locally symmetric, hence locally homogeneous; therefore, we need consider only nonsymmetric left-invariant metrics here.  A consequence of their Theorems 3.5, 3.6, and  3.7 is that the Ricci operator  of a left-invariant, nonsymmetric Lorentz metric on a {\it nonunimodular} three-dimensional Lie group admits a triple eigenvalue $\lambda$ if and only if $\lambda =0$, and the Ricci operator is nilpotent of order two.  We conclude $\lambda =0$, so  $A(s)$ has the form claimed.  Moreover, $A$ is nilpotent of order two on $U \setminus S$.
      
\smallskip
   (ii)     Because $\lieg$ acts transitively on $S$, there is an adapted framing along $S$ in which $A \equiv A(s)$.
 The parameter $\alpha$ in $A(s)$ cannot vanish; otherwise the curvature of $g$ vanishes on $S$ and 
      $(S, \nabla)$ is flat, which was proved to be impossible in Proposition~\ref{surface}.
      Now the Ricci operator on $S$ is nontrivial and lies  in the closure of the $PSL(2,\RR)$-orbit $\mathcal{O}$ of the Ricci operator  on $U \setminus S$.  But we know from (i) that on $U \backslash S$, the Ricci operator is $g$-symmetric and nilpotent of order 2, so it has the same form as $A(s)$, meaning it also belongs to the $PSL(2,\RR)$-orbit of      
       $$  \left(  \begin{array}{ccc}
                                                                 0    &    0   & 1 \\
                                                                 0     &  0  &    0     \\
                                                                 0     &   0  &  0 \\
                                                                 \end{array} \right).$$
%The isotropy of the Ricci operator is an one parameter unipotent  subgroup $H$  in $PSL(2,\RR)$ (which coincides on $S$ with the isotropy of $g$).  
 \end{proof}
Now $\mbox{Ricc}(u,u)$ is a quadratic form of rank one equal to $g(W,u)^2$, for some nonvanishing isotropic vector field $W$ on $U$, which coincides with $X$ on $S$.  
% and  $H$ is the stabilizer of $W$ in the structural group $PSL(2,\RR)$  of the orthonormal frame bundle $R(U)$.The reduction of the structural subgroup of the orthonormal frame bundle $R(U)$  to the unipotent subgroup $H$  and the vector field $W$  are   $\mathfrak{g}$-invariant. 
Invariance of $\operatorname{Ricci}$ by $\lieg$ implies invariance of $W$.  Proposition~\ref{invariant theory}  implies that $g$ is locally homogeneous.

\section{Alternate proofs using the Cartan connection} \label{section5}

The aim of this section is to give a second proof of Theorem~\ref{main} using the Cartan connection associated to a Lorentz metric.  The reader can find more details about the geometry of Cartan connections in the book~\cite{Sharpe}.  We still consider $g$ a Lorentz metric defined in a connected  open neighborhood $U$ of the origin in $\RR^3$.

\label{sec.cartan.connxn}

\subsection{Introduction to the Cartan connection}

 Let $\lieh = \mathfrak{o}(2,1) \ltimes \RR^{2,1}$.  Let $P = O(2,1) < O(2,1) \ltimes \RR^{2,1}$, so $\liep = \mathfrak{o}(2,1) \subset \lieh$.   Let $\pi : B \rightarrow U$ be the principal $P$-bundle of normalized frames on $U$, in which the Lorentz metric $g$ has the matrix form
$$ \mathbb{I} = 
\left(
\begin{array}{ccc}
  &   & 1 \\
   & 1   &  \\
1 &   & 
\end{array}
\right).
$$
(Note that $B$ is nearly the same as the bundle $R(U)$ from Section \ref{section2}, though it has been enlarged to allow all possible orientations and time orientations.)

The \emph{Cartan connection} associated to $(U,g)$ is the 1-form $\omega \in \Omega^1(B,\lieh)$ formed by the sum of the Levi-Civita connection of the metric $\nu \in \Omega^1(B,\liep)$ and the tautological 1-form $\theta \in \Omega^1(B, \RR^{2,1})$, defined by $\theta_b(v) = b^{-1}(\pi_* v)$.  The form $\omega$ satisfies the following axioms for a Cartan connection:
\begin{enumerate}
\item It gives a parallelization of $B$---that is, for all $b \in B$, the restriction $\omega_b : T_b B \rightarrow \lieh$ is an isomorphism.

\item It is $P$-equivariant: for all $p \in P$, the pullback $R_p^* \omega = \Ad p^{-1} \circ \omega$.

\item It recognizes fundamental vertical vector fields: for all $X \in \liep$, if $X^\ddag$ is the vertical vector field on $B$ generated by $X$, then $ \omega(X^\ddag) \equiv X$.
\end{enumerate}

 The \emph{Cartan curvature} of $\omega$ is
$$ K (X,Y) = \mbox{d} \omega(X,Y) + [\omega(X),\omega(Y)].$$
This 2-form is always semibasic, meaning $K_b(X,Y)$ only depends on the projections of $X$ and $Y$ to $T_{\pi(b)}U$; in particular, $K$ vanishes when either input is a vertical vector.  We will therefore express the inputs to $K_b$ as tangent vectors at $\pi(b)$.  Torsion-freeness of the Levi-Civita connection implies that $K$ has values in $\liep$.  Thus $K$ is related to the usual Riemannian curvature tensor $R \in \Omega^2(U) \otimes \mbox{End}(TM)$ by 
$$ b \circ R_{\pi(b)} (u, v) \circ b^{-1}  = K_b(u,v).$$
The benefit here of working with the Cartan curvature is that, when applied to Killing vector fields, it gives a precise relation between the brackets on the manifold $U$ and the brackets in the Killing algebra $\lieg$. 

The $P$-equivariance of $\omega$ leads to $P$-equivariance of $K$: $(R_p^* K)(X,Y) = (\Ad p^{-1})(K(X,Y))$.  The infinitesimal version of this statement is, for $A \in \liep$,
$$ K([A^\ddag,X],Y) + K(X,[A^\ddag,Y]) =  [K(X,Y), A].$$

A Killing field $Y$ on $U$ lifts to a vector field on $B$, which we will also denote $Y$, with $L_Y \omega= 0$.  Note that also $L_Y K = 0$ in this case.  Thus if $X$ and $Y$ are Killing fields, then
$$ X. (\omega (Y)) = \omega[X,Y] \qquad \mbox{and} \qquad Y.(\omega(X)) = \omega[Y,X].$$
In this case,
\begin{eqnarray*}
K(X,Y) & = & X.(\omega(Y)) - Y. (\omega(X)) - \omega[X,Y] + [\omega(X),\omega(Y)] \\
& = & \omega[X,Y] - \omega[Y,X] - \omega[X,Y] + [\omega(X),\omega(Y)] \\
& = & \omega[X,Y] + [\omega(X),\omega(Y)]
\end{eqnarray*}
so, when $X$ and $Y$ are Killing, then
\begin{eqnarray}
\label{eqn.curv.killing}
\omega[X,Y] = [\omega(Y),\omega(X)] + K(X,Y). 
\end{eqnarray}

Via the parallelization given by $\omega$, the semibasic, $\liep$-valued 2-form $K$ corresponds to a $P$-equivariant, automorphism-invariant function 
$$ \kappa : B \rightarrow \wedge^2 \RR^{2,1*} \otimes \liep.$$
The $P$-representation on the target vector space is associated naturally to the adjoint representation of $G$ restricted to $P$, and will be denoted $g \cdot \kappa(b)$, for $g \in P$ and $b \in B$.  We will use the same notation below for other $P$-represenations associated to the adjoint, and also for the corresponding Lie algebra representations---for example, $X \cdot \kappa(b)$ for $X \in \liep$.

\subsection{Curvature representation}
\label{subsection.curv.rep}

Denote $(e,h,f )$ a basis of $\RR^{2,1}$ in which the inner product is given by $\mathbb{I}$. 
%$$ \mathbb{I} = 
%\left(
%\begin{array}{ccc}
%  &   &   1 \\
%  & 1 &   \\
%  1 & & 
%  \end{array}
%  \right)
%$$
Let $E,H,F$ be generators of $\liep$ with matrix expression in the basis $(e,h,f )$
\begin{eqnarray*}
E = 
\left(
\begin{array}{ccc}
0  & - 1  &   \\
  & 0   & 1 \\
  &    & 0
  \end{array}
  \right)
   \qquad 
  H = 
  \left(
  \begin{array}{ccc}
  1 &  & \\
     & 0 &  \\
     &   & -1
     \end{array}
     \right)
        \qquad 
     F = 
     \left( 
     \begin{array}{ccc}
     0  &   &   \\
       - 1  & 0  & \\
       &   1 & 0
       \end{array}
       \right).
              \end{eqnarray*}
Therefore this representation of $\liep$ is equivalent to $\mbox{ad } \liep$ via the isomorphism sending $(e,h,f )$ to $( E, H, F)$.

%Denote by $\{ e^*, h^*, f^* \}$ the dual basis of $\{ e,h,f \}$, determined by the property that the row vectors for $e^*, h^*,$ and $f^*$ in terms of the basis $\{ e,h, f \}$ form the identity matrix.  Thus $e^*(x) = \langle f,x \rangle$, $h^*(x) = \langle h,x \rangle$, and $f^*(x) = \langle e,x \rangle$.  
Denote by $\ ^*$ the isomorphism $\RR^{2,1} \rightarrow \RR^{2,1*}$ with $w^*(u) = \langle w,u \rangle$.  Note that for $p \in O(2,1)$ and $x \in \RR^{2,1}$, we have $(px)^* = p^*x^*$ for the dual represention $p^*x^* = x^* \circ p^{-1}$.

%% Now $E \mapsto e$,  $H \mapsto h$,  $F \mapsto f$ defines an $\mathfrak{o}(2,1)$-equivariant isomorphism from $\mathfrak{o}(2,1)$ to $\RR^{2,1}$---that is, for each $X \in \mathfrak{o}(2,1)$, the matrix of $\mbox{ad }X$ in the basis $E,H,F$ equals the matrix of $X$ in the basis $e,h,f$.

%% Next consider the volume form $vol = e^* \wedge h^* \wedge f^*$, which is invariant by the dual action of $\mbox{SO}(2,1)$.   Then define an isomorphism $\tau : \wedge^2 \RR^{2,1*} \rightarrow \RR^{2,1*}$ by 

%% $$vol = \mu \wedge \langle \tau(\mu)^*, \  \rangle$$

%% It is straightforward to see that this map is an equivariant linear isomorphism: $\tau(p^* \mu) = p^* \tau(\mu)$ for all $p \in \mbox{SO}(2,1)$.  In terms of a basis, 
%% \begin{eqnarray*}
%% \tau & : & e^* \wedge h^* \mapsto e^* \\
%%  &  & f^* \wedge e^* \mapsto  h^* \\
%%  &   & h^* \wedge f^* \mapsto f^*
%%  \end{eqnarray*}

% The tensor product of these isomorphisms $\wedge^2 \RR^{2,1*} \rightarrow \RR^{2,1*}$ and $\mathfrak{o}(2,1) \rightarrow \RR^{2,1}$ gives

Next we define an $O(2,1)$-equivariant homomorphism $\varphi : \wedge^2 \RR^{2,1*} \otimes \mathfrak{o}(2,1)  \rightarrow \RR^{3*} \otimes \RR^3$, where the representation on $\mbox{End } \RR^3$ is by conjugation.
%(We will also denote the $P$ and $\liep$ representations on this latter space by $g \cdot$ and $X \cdot$, respectively.)
% In fact, $\wedge^2 \RR^{2,1*} \otimes \mathfrak{o}(2,1)$ is isomorphic as an $\mbox{SO}(2,1)$-module to $\RR^{3 \times 3}$.
Define $\varphi$ on simple tensors by   
$$ \varphi(v^* \wedge w^* \otimes X) =  (Xv)^* \otimes w - (Xw)^* \otimes v  =  (w^* \circ X) \otimes v - (v^* \circ X) \otimes w. $$
Equivariance is easy to check.  When the input lies in the submodule $\mathbb{W}$ satisfying the Bianchi identity, then the output is $\mathbb{I}$-symmetric (see \cite{Sharpe}, Section 6, Proposition 1.4 (ii)(c)).  The Ricci endomorphism $A$, defined in terms of the curvature tensor by 
$$ \langle A_xv, w \rangle = \mbox{tr } R_x(v, \cdot)w = \mbox{Ricci}_x(v,w), \qquad \forall v,w \in T_xM$$
corresponds via $\omega$ to the function $\varphi \circ \kappa$. Recall that in dimension 3, the curvature tensor is determined by the Ricci curvature, so $\varphi$ restricted to $\mathbb{W}$ is actually an isomorphism onto its image.

This image is the sum $E_0 \oplus E_2$  of two irreducible components of the $O(2,1)$-representation on $\mbox{End } \RR^3$.  The first, denoted $E_0$, is the one-dimensional trivial representation, generated by the identity on $\RR^3$, which we will denote $m_d$.  Another irreducible component $E_1$ corresponds to endomorphisms in $\mathfrak{o}(2,1)$, 
% A basis for $E_1$ corresponding to the matrices $E,H,$ and $F$ above will be denoted $\{ m_U, m_A, m_V \}$.  
%Now consider the polar decomposition of $\mbox{End } \RR^3$ with respect to the Lorentzian inner product.  The subspace $E_1$ comprises endomorphisms $X$ 
which satisfy $X {\mathbb I} = - {\mathbb I} X^t$.  The $O(2,1)$-invariant complementary subspace, consisting of the $\mathbb{I}$-symmetric endomorphisms, splits into $E_0$ and the last irreducible component, $E_2$, which is five-dimensional.  The component $E_0$ captures the scalar curvature, while $E_2$ corresponds to the tracefree Ricci endomorphism. 
%more precisely, our realization of $E_0 \oplus E_2$ as a representation contained in $\mbox{End } \RR^3$ gives Ricci endomorphisms. , so we will focus below on the components of the curvature in $E_0 \oplus E_2$.  

In the second column of the following table, we list a basis for $E_0 \oplus E_2$, with notation for each element in the first column, and the elements of $\mathbb{W} \subset \wedge^2 \RR^{2,1*} \otimes \mathfrak{o}(2,1)$ mapping to them under $\varphi$ in the third column.  Note that the elements in the last column span the space of all possible values of $\kappa$.

\bigskip

\begin{tabular}{|c | c | c |}
\hline
  & $\RR^{3\times 3}$  & $\mathbb{W} \subset \wedge^2 \RR^{2,1*} \otimes \mathfrak{o}(2,1)$  \\
  \hline
$m_d$ & $ 2(f^* \otimes e + h^* \otimes h + e^* \otimes f)$ &  $h^* \wedge e^* \otimes F + e^* \wedge f^* \otimes H + f^* \wedge h^* \otimes E$ \\
%  $m_E$  & $- h^* \otimes e + f^* \otimes h$ & $e^* \wedge f^* \otimes E + h^* \wedge f^* \otimes H$ \\
 % $m_H$ & $e^* \otimes e - f^* \otimes f$ & $e^* \wedge h^* \otimes E + f^* \wedge h^* \otimes F$  \\
% $m_F$ & $- e^* \otimes h + h^* \otimes f$ & $h^* \wedge e^* \otimes H + f^* \wedge e^* \otimes F$ \\
 $m_{e^2}$ & $ e^* \otimes e$ & $e^* \wedge h^* \otimes E$ \\
 $m_{eh}$ & $h^* \otimes e + e^* \otimes h$ & $f^* \wedge e^* \otimes E + f^* \wedge h^* \otimes H$ \\
 $m_{2h^2-ef}$ & $2 h^* \otimes h - f^* \otimes e -  e^* \otimes f$ & $2 f^* \wedge e^* \otimes H + f^* \wedge h^* \otimes E + h^* \wedge e^* \otimes F$ \\ 
$m_{hf}$ & $f^* \otimes h + h^* \otimes f$ & $h^* \wedge f^* \otimes H + f^* \wedge e^* \otimes F$  \\
$m_{f^2}$ & $f^* \otimes f$ & $h^* \wedge f^* \otimes F$ \\
\hline
\end{tabular}

\bigskip

% The $U$-invariant elements here all lie in the span of $m_d$ and $m_{u^2}$, corresponding to the parameters $\gamma$ and $\alpha$ in our current writeup.  The $U$-invariant elements of the full representation $E$ comprise the span of $m_d$, $m_U$ and $m_{u^2}$.  The tangent space to $S$ evaluates in this framing to span$\{ u,a \}$.  Both $m_U$ and $m_{u^2}$ vanish on this subspace.  Therefore, our surface with unipotent isotropy is flat if and only if $\gamma = 0$.

Assume now that $g$ is quasihomogeneous.
Recall  that, by the results in Section~\ref{section2}, the Killing algebra $\mathfrak{g}$ is three-dimensional. It acts transitively on $U$, away from a two-dimensional, degenerate submanifold $S$ passing through the  origin.  Moreover, $\mathfrak{g}$ acts transitively on $S$ and  the isotropy at points of $S$ is conjugated to a one-parameter semisimple group or to a one-parameter unipotent group in $PSL(2, \RR)$.  We will study the interaction of $\lieg$, $\omega (\lieg)$, and $\kappa$, both on and off $S$.

\subsection{Semisimple isotropy}

Let $b_0$ be a point of $B$ lying over the origin and assume that the isotropy action of $\lieg$ at $0$ is semisimple, as in Section \ref{section3}.  A semisimple element of $\liep$ is conjugate in $P$ into $\RR H$, so up to changing the choice of $b_0 \in \pi^{-1}(0)$, we may assume that $\omega_{b_0} (\lieg) \cap \liep$ is spanned by $H$.  

\begin{proposition}(compare Lemma \ref{Killing} ($i$))
\label{prop.aff.plus.r}
If  the isotropy of $\lieg$ at the origin is semisimple, then $\lieg \cong \RR \oplus \mathfrak{aff}(\RR)$.
\end{proposition}

\begin{proof}
Let $Y \in \lieg$ have $\omega_{b_0} (Y) = H$, so the corresponding Killing field vanishes at the origin.  The projection $\overline{\omega_{b_0}(\lieg)}$ of $\omega_{b_0}(\lieg)$ to $\RR^{2,1}$ is 2-dimensional, degenerate, and $H$-invariant.  Again, by changing the point $b_0$ in the fiber above the origin, we may conjugate by an element normalizing $\RR H$ so that this projection is span$\{ e,h \}$.  Therefore, there is a basis $(X,Y,Z )$ of $\lieg$ such that

$$ \omega_{b_0}(X) = h + \alpha E + \beta F \qquad \mbox{and} \qquad \omega_{b_0} (Z) =  e + \gamma E + \delta F$$
for some $\alpha, \beta, \gamma, \delta \in \RR$.  Because $K_{b_0} (Y, \cdot ) = 0$, equation (\ref{eqn.curv.killing})   gives
$$ \omega_{b_0} [Y,X] = [h + \alpha E + \beta F, H] = - \alpha E + \beta F \in \omega_{b_0} (\lieg)$$
so $\alpha = \beta = 0$ and $[Y,X] = 0$.  A similar computation gives 
$$\omega_{b_0} [Y,Z] =  [ e + \gamma E + \delta F,H] = - e - \gamma E + \delta F$$ 
so $\delta = 0 $, and $[Y,Z] = - Z$.

Infinitesimal invariance of $K$ by $Y$ gives
$$ K_{b_0}([Y,X],Z) + K_{b_0}(X,[Y,Z]) = Y.( K(X,Z))_{b_0} = H^\ddag. (K(X,Z))_{b_0} = [- H, K_{b_0}(X,Z)],$$
which reduces to $ K_{b_0}(X, Z) = [H, K_{b_0}(X, Z)].$
Since $K$ takes values in $\mathfrak{p}$, where $\mbox{ker}(\ad H - \mbox{Id}) = \RR E$, we get 
$$ K_{b_0}(X, Z)= \kappa_{b_0} (h,e) =  r E \qquad \mbox{for some }  r \in \RR.$$

Now equation (\ref{eqn.curv.killing}) gives for $X$ and $Z$,
\begin{eqnarray*}
\omega_{b_0} [X,Z] & = & [ e + \gamma E, h]  + rE  \\
& = & - \gamma e + rE.
\end{eqnarray*}
In order that this element belong to $\omega_{b_0}(\lieg) = \mbox{span} \{ H, h, e + \gamma E \}$, we must have $r = - \gamma^2$, and $\lbrack X, Z \rbrack =- \gamma Z$.

 The structure of the algebra $\lieg$ in the basis $(X, Y, Z )$ is
\begin{eqnarray*}
\mbox{ad } Y = 
\left( 
\begin{array}{ccc}
0 &  &   \\
   &    0 &  \\
   &   & -1
   \end{array}
      \right)
   \qquad
   \mbox{ad } X = 
   \left(
   \begin{array}{ccc}
   0 &   &  \\
     & 0 &   \\
     &   & - \gamma 
     \end{array}
     \right)
     \qquad
     \mbox{ad } Z = 
     \left(
     \begin{array}{ccc}
    0 &  & \\
      & 0 &  \\
    \gamma  &  1 & 0
      \end{array}
      \right).
      \end{eqnarray*}
This $\lieg$ is isomorphic to $\mathfrak{aff}(\RR) \oplus \RR$, with center generated by $\gamma Y - X.$ 
\end{proof}

 Let $W = X- \gamma Y$.   Note that $W(0)$ has norm 1 because $\overline{\omega_{b_0} (W) }= h$.  As in Section \ref{section3}, where the central element of $\lieg$ is called $X'$, the norm of $W$ is constant 1 on $U$ because it is $\lieg$-invariant and equals $1$ at a point of $S$.  Existence of a Killing field of constant norm 1 has the following consequences for the geometry of $U$:

\begin{proposition} 
\ \ 
\begin{enumerate}[(i)]
\item  The local $\lieg$-action on $U$ preserves a splitting of $TU$ into three line bundles, $L^- \oplus \RR W \oplus L^+$, with $L^-$ and ${L}^+$ isotropic.

\item The distributions $L^- \oplus \RR W$ and $L^+\oplus \RR W$ are each tangent to $\lieg$-invariant, degenerate, totally geodesic foliations $\mathcal{P}^-$ and $\mathcal{P}^+$, respectively; moreover, the surface $S$ is a leaf of one of these foliations, which we may assume is $\mathcal{P}^+$. 
\end{enumerate}
\end{proposition}

\begin{proof}
(i)  Because $\lieg$ preserves $W$, it preserves $W^\perp$, which is a 2-dimensional Lorentz distribution.  A 2-dimensional Lorentz vector space splits into two isotropic lines preserved by all linear isometries.  Therefore $W^\perp = {L}^- \oplus {L}^+$, with both line bundles isotropic and  $\lieg$-invariant.

\smallskip

(ii)  Because the flow along $W$ preserves ${L}^-$ and ${L}^+$, the distributions  ${L}^- \oplus \RR W$ and ${L}^+ \oplus \RR W$ are involutive, and thus they each integrate to foliations $\mathcal{P}^-$ and $\mathcal{P}^+$ by degenerate surfaces.

Let $x \in U$.  Let $V^- \in \Gamma(L^-)$ and  $V^+ \in \Gamma(L^+)$ be vector fields with $V^\pm(x) \neq 0$ and $[W,V^\pm](x) = 0$.
% Let $x \in U \backslash S$, so that $\lieg(x) = T_x U$.  Choose nonzero $v^- \in {L}^-(x)$ and $v^+ \in {L}^+(x)$.  
%Because $\lieg$ acts simply transitively on a neighborhood of $x$ preserving both line bundles, these vectors can be translated by flows from $\lieg$ to yield vector fields $V^- \in \Gamma(L^-)$ and  $V^+ \in \Gamma(L^+)$ on a neighborhood of $x$.  Because $W$ is central in $\lieg$, the brackets $[W,V^-]$ and $[W,V^+]$ vanish.  
It is well known that a Killing field of constant norm is geodesic: $\nabla_W W = 0$.  Moreover,  because $g(V^\pm , V^\pm)$ is constant zero, $W. (g(V^\pm, V^\pm)) = V^\pm. (g( V^{\pm}, V^\pm)) = 0$, from which
$$ g_x(\nabla_W V^{\pm}, V^\pm) = g_x(\nabla_{V^\pm} W ,V^\pm) = g_x(\nabla_{V^\pm} V^\pm, V^\pm) = 0.$$
The tangent distributions $T\mathcal{P}^\pm$ equal $(V^\pm)^\perp$, and it is now straightforward to verify from the axioms for $\nabla$ that $\mathcal{P}^-$ and $\mathcal{P}^+$ are totally geodesic through $x$.   

The Killing field $W$ is tangent to the surface $S$.  Because $S$ is degenerate, $TS^\perp$ is an isotropic line of $W^\perp$ and therefore coincides with $L^+$ or $L^-$.  We can assume it is $L^+$, so $S$ is a leaf of $\mathcal{P}^+$; in particular, we have shown $S$ is totally geodesic.
\end{proof}

\begin{proposition}
\label{prop.curv.ann}
\ \
\begin{enumerate}[(i)]
\item For $x \in U$ and $u,v \in T\mathcal{P}^\pm_x$, the curvature $R_x(u,v)$ annihilates $(\mathcal{P}_x^\pm)^\perp$.

\item The Ricci endomorphism at $x$  preserves each of the line bundles $L^+, \RR W$, and $L^-$.
\end{enumerate}
\end{proposition}

\begin{proof}
(i) The argument is the same for $\mathcal{P}^+$ and $\mathcal{P}^-$, so we write it for $\mathcal{P}^-$.  Let $x \in U \backslash S$.  Because $\lieg$ acts transitively on a neighborhood of $x$, there is a Killing field $A^-$ evaluating at $x$ to a nonzero element of $L^-(x)$.   Note that $[A^-,W] = 0$.  The orbit of $x$ under $A^-$ and $W$ coincides near $x$ with an open subset of $\mathcal{P}^-_x$. Because $L^-$ is $\lieg$-invariant, the values of $A^-$ in this relatively open set belong to $L^-$.  

Now $A^-. (g( A^- , A^- )) = 0$ implies $g( \nabla_{A^-} A^-, A^- ) = 0$, and $A^-. (g( A^-, W )) = 0$ gives
$$ 0 = g _x(\nabla_{A^-} A^-, W ) + g_x( A^-, \nabla_{A^-} W ) = g_x( \nabla_{A^-} A^-, W ),$$
using that $\mathcal{P}_x^-$ is totally geodesic.  Therefore $(\nabla_{A^-} A^-)_x = a A^-$ for some $a \in \RR$.  The flows along $A^-$ and $W$ act locally transitively on $\mathcal{P}^-_x$ preserving the connection $\nabla$ and commuting with $A^-$.  Thus $\nabla_{A^-} A^- \equiv a A^-$ on a neighborhood of $x$ in $\mathcal{P}^-_x$.

Next, $W. (g( A^- , W )) = 0$ gives 
$$ 0 = g( \nabla_W A^-, W ) + g( A^-, \nabla_W W ) =g( \nabla_W A^-, W ),$$
using that $W$ is geodesic.  Therefore $(\nabla_W A^-)_x = b A^-$ for some $b \in \RR$.  Again invariance of $\nabla$, $A^-$, and $W$ implies that $\nabla_W A^- \equiv b A^-$ on a neighborhood of $x$ in $\mathcal{P}^-_x$.  Now we compute
$$ R_x(A^-,W) A^- = (\nabla_{A^-} \nabla_W  - \nabla_W \nabla_{A^-}  - \nabla_{[A^-,W]})A^- = \nabla_{A^-} (bA^-) - \nabla_W (aA^-) = ab A^- - ba A^- = 0.$$
This property of the curvature we have proved on $U \backslash S$ remains true on $S$ because it is a closed condition.

\smallskip

(ii) It suffices to show that the Ricci endomorphism  preserves $L^- \oplus \RR W = T\mathcal{P}^-$ and $L^+ \oplus \RR W = T \mathcal{P}^+$.  Then invariance of $L^+$ and $L^-$ will follow from symmetry of $A$ with respect to $g$.  Again, we just write the proof for $\mathcal{P}^-$.  The Ricci endomorphism preserves $T \mathcal{P}^-$ if and only if $\mbox{Ricci}_x(u,v) = \mbox{Ricci}_x(v,u) = 0$ for any $u \in L_x^-$, $v \in T \mathcal{P}^-_x$.  Assume $u \neq 0$ and complete it to an adapted basis $( u, w, z )$ of $T_xU$ with $w = W(x)$, $z \in L^+_x$, and $g_x(u,z) = 1$.  Then, by part (i),
$$ \mbox{Ricci}_x(v,u) = g_x( R(v,u)u,z ) + g_x( R(v,w)u,w ) + g_x( R(v,z)u,u ) = 0+0+0=0.$$
\end{proof}

Let $\mathcal{R}$ be the $\lieg$-invariant reduction of $B$ to the subbundle comprising frames $(x, (v^-, W(x), v^+))$ with $v^- \in L_x^-$ and $v^+ \in L_x^+$.  Now $\mathcal{R}$ is a principal $A$-bundle, where $\RR^* \cong A < P$ is the subgroup with matrix form
$$ A = \left\{ 
\left(
\begin{array}{ccc}
\lambda^2 &   &   \\
  & 1 &   \\
 &  & \lambda^{-2}
\end{array}
\right)
\ : \ \lambda \in \RR^*
\right\}.
$$

Note that, at any $b \in \mathcal{R}$, the projection $\overline{\omega_b(W)} = h$.  Proposition \ref{prop.curv.ann} translates to the following statement on $\mathcal{R}$.

\begin{proposition}
For any $b \in \mathcal{R}$, the component $\bar{\kappa}_b$ in the representation $E_0 \oplus E_2$, corresponding to the Ricci endomorphism, is diagonal, so has the form
$$\bar{\kappa}_b = ym_d + zm_{2h^2 - ef} \qquad y, z \in \RR.$$
\end{proposition}
Note that $H \cdot \bar{\kappa}_b = 0$, so by $P$-equivariance of $\bar{\kappa}$, the derivative in the vertical direction $H^\ddag.\bar{\kappa}_b=0$.  Because this curvature function is also $\lieg$-invariant, it is constant on $\mathcal{R}_{U \backslash S}$.  By continuity, we conclude that on all $\mathcal{R}$,
$$ \bar{\kappa} \equiv ym_d + zm_{2h^2 - ef} \qquad y, z \in \RR.$$

%On the other hand, invariance of $\bar{\kappa}_{b_0}$ by the semisimple isotropy (generated by $H$)  at the origin means $\bar{\kappa}_{b_0} \in \mbox{span}\{ m_d, m_{2h^2 -ef} \}$.   
%
Since $\mathfrak{g}$ acts transitively on $U \setminus S$ and preserves $\mathcal{R}$,  for any $b \in \left.\mathcal{R}\right|_{U \backslash S}$
there exists  a sequence $a_n$ in $A$ such that $\varphi_n b a_n^{-1} \rightarrow b_0$, where each $\varphi_n$ is in the pseudo-group generated by flows along local Killing fields in $\lieg$; then $(\mbox{Ad } a_n)(\omega_b(\lieg)) \rightarrow \omega_{b_0}(\lieg)$ in the Grassmannian $\mbox{Gr}(3,\lieh)$.  Let us consider such a sequence $a_n$ corresponding to a point $b \in B$ lying above  $U \setminus S$. Then we prove the following

\begin{lemma}
\label{lemma.lambda.diverges}
Write
$$a_n = 
\left(
\begin{array}{ccc}
\lambda_n^2 &   &   \\
  & 1 &   \\
 &  & \lambda_n^{-2}
\end{array}
\right)
, \  \lambda_n \in \RR^*.
$$
Then $\lambda_n \rightarrow \infty$.
\end{lemma}

\begin{proof}
First note that $\lambda_n$ cannot  converge to a nonzero number, because in this case $\lim_n (\Ad a_n)(\omega_b(\lieg)) = \omega_{b_0}(\lieg)$ would still project onto $\RR^{2,1}$ modulo $\liep$, contradicting that the $\lieg$-orbit of $0$ is two-dimensional.  This also shows that $a_n$ cannot admit a convergent  subsequence, meaning that $a_n$ goes to infinity in $A$.

The space $\omega_b(\lieg)$ can be written as $\mbox{span} \{ e + \rho(e), h + \rho(h), f + \rho(f) \}$ for $\rho : \RR^{2,1} \rightarrow \liep$ a linear map.  The space $(\Ad a_n)(\omega_b(\lieg))$ contains $\lambda_n^{-2} f + a_n \cdot \rho(f)$, so it contains $f + \lambda_n^2 a_n \cdot \rho(f)$.  If $\lambda_n \rightarrow 0$, then this last term converges to $f + \xi \in \omega_{b_0}(\lieg)$, for some $\xi \in \liep$ (because the adjoint action of $a_n$ on $\mathfrak{p}$ is diagonal with eigenvalues $\lambda_n^2$, 1 and $\lambda_n^{-2}$).  But $\overline{\omega_{b_0} (\lieg)}$ is spanned by $e$ and $h$, so this is a contradiction.
\end{proof}

Differentiating the function $\bar{\kappa} : B \rightarrow \mathbb{V}^{(0)} = E_0 \oplus E_2$ gives, via the parallelization of $B$ arising from $\omega$, a $P$-equivariant, automorphism-invariant function $D^{(1)} \bar{\kappa}: B \rightarrow \mathbb{V}^{(1)} = \lieh^* \otimes \mathbb{V}^0$, and similarly, by iteration, functions $D^{(i)} \bar{\kappa} : B \rightarrow \mathbb{V}^{(i)} = \lieh^* \otimes \mathbb{V}^{(i-1)}$; automorphism-invariant here means $D^{(i)} \bar{\kappa}(f(b)) = D^{(i)} \bar{\kappa}(b)$ for all $b \in B$ and all automorphisms $f$.  For vertical directions $X \in \liep$, the derivative is determined by equivariance: $X^\ddag. \bar{\kappa} = - X \cdot \bar{\kappa}$. Our goal, in order to show local homogeneity of $U$, is to show that $D^{(i)} \bar{\kappa}$ has values on $B$ in a single $P$-orbit.  Because $\bar{\kappa}$ determines $\kappa$ for 3-dimensional metrics, it will follow that $D^{(i)} \kappa$ has values on $B$ in a single $P$-orbit, which suffices by Singer's theorem  to conclude local homogeneity (see
Proposition 3.8 in~\cite{Melnick2} for a version of Singer's theorem for real analytic Cartan connections and  also~\cite{Pecastaing} for the smooth case). By $P$-equivariance of these functions, it suffices to show that the values on $\mathcal{R}$ lie in a single $A$-orbit.  We will prove the following slightly stronger result:

\begin{proposition}
The curvature $\bar{\kappa}$ and all of its derivatives $D^{(i)} \bar{\kappa}$ are constant on $\mathcal{R}$. 
\end{proposition}

\begin{proof}
Recall that
%By the lemma above, $a_n  m_{e^2} = \lambda_n^4 m_{e^2}$ and $a_n m_{eh} = \lambda_n^2 m_{eh}$ are both unbounded as $n \rightarrow \infty$.  Therefore $\bar{\kappa}_b$ cannot have nonzero components on $m_{e^2}$ or $m_{eh}$, and $\bar{\kappa}_b = y m_d + z m_{2h^2 - ef}$ for some $y,z \in \RR$.  Then 
%$$ a_n \bar{\kappa}_b = \bar{\kappa}_b \rightarrow \bar{\kappa}_{b_0} = y m_d + z m_{2h^2 - ef} $$
$$\bar{\kappa} \equiv y m_d + z m_{2h^2 - ef}$$
on all of $\mathcal{R}$, for some fixed $y,z \in \RR$.
%Now $H \cdot \bar{\kappa}_b = 0$.  Since $\mathcal{R}$ and $\bar{\kappa}$ are $\lieg$-invariant, $\bar{\kappa}$ is also $H$-invariant on every fiber above  $U \backslash S$ in $\mathcal{R}$.  The Killing fields of $\lieg$ together with $H^\ddag$ span $T\mathcal{R}$, so $\bar{\kappa}$ is constant over all points of $\left. \mathcal{R} \right|_{U \backslash S}$ and therefore on all of $\mathcal{R}$.
The proof proceeds by induction on $i$.  Suppose that for $i \geq 0$, the derivative $D^{(i)} \bar{\kappa}$ is constant on $\mathcal{R}$, so that in particular, the value $D^{(i)} \bar{\kappa}$ is annihilated by $H$.  As in the proof for $i=0$ above, to show that $D^{(i+1)} \bar{\kappa}$ is constant on $\mathcal{R}$, it suffices to show that $H^\ddag. D^{(i+1)} \bar{\kappa}_b = - H \cdot D^{(i+1)} \bar{\kappa}_b = 0$ at a single point $b \in \left. \mathcal{R} \right|_{U \backslash S}$.

To complete the induction step, we will need the following information on $\omega_b(\lieg)$.

\begin{lemma}
\label{lemma.omega.g}
At $b \in \mathcal{R}$ lying over $x \in U \backslash S$, the Killing algebra evaluates to
$$ \omega_b(\lieg) = \mbox{span} \{ e + \gamma E + \beta H, h  - \gamma H, f + \alpha H + \delta F \}, \qquad \gamma, \beta, \alpha, \delta \in \RR.$$
% Here, $\gamma$ is as in proposition \ref{prop.aff.plus.r}, while $\beta, \alpha,$ and $\delta$ are new arbitrary real constants.
\end{lemma}

\begin{proof}
Write 
$$\omega_b(\lieg) = \mbox{span} \{ e + \rho(e), h + \rho(h), f + \rho(f) \}.$$
From Proposition \ref{prop.aff.plus.r}, we know that
$$ (\Ad a_n)(\omega_b(\lieg)) \rightarrow \omega_{b_0}(\lieg) = \mbox{span}\{ e + \gamma E,h,H \}.$$
Now Lemma \ref{lemma.lambda.diverges} implies that $\rho(h)$ and $\rho(f)$ both have zero component on $E$. Indeed, since  this component is dilated by $\lambda_n^2$,  it must vanish in order that $E \notin \omega_{b_0}(\lieg)$.  

At the point $b$, let $A^-$ be a Killing field with $\pi_{*b} A^- \in L^-_{\pi(b)}$, so we can assume $\overline{\omega_b(A^-)} = e$. We have ${\omega}_b(A^-) = e + \rho(e)$ and $\omega_b(W) = h + \rho(h)$.  
Recall from Proposition \ref{prop.aff.plus.r} that $\kappa_{b_0}(h,e) = r E$. The fact that $\bar{\kappa}_b = \bar{\kappa}_{b_0}$ implies that the full curvature $\kappa_b = \kappa_{b_0}$, so also 
$$ \kappa_b(e,h) = K_b(A^-,W) = r E.$$
On the other hand, equation (\ref{eqn.curv.killing}) gives
$$ 0 = \omega_b[A^-,W]  =  [h + \rho(h), e + \rho(e)]  + r E,$$
so 
$$ \rho(h) e = \rho(e)h \qquad \mbox{and} \qquad [\rho(h), \rho(e)] = -r E.$$ 
Writing $\rho(e) = \gamma E + \beta H + \delta F$ and $\rho(h) = \beta' H + \delta' F$ gives $\beta' = - \gamma$ and $\delta = \delta' = 0$ from the first equation.   Note that the second equation gives $\gamma^2 = -r$, which is consistent with Proposition \ref{prop.aff.plus.r}.  
 \end{proof}

We now use $\lieg$-invariance of $D^{(i)}\bar{\kappa}$.  For abitrary $X \in \lieh$, write $X^\ddag$ for the coresponding $\omega$-constant vector field on $B$.  Lemma \ref{lemma.omega.g} gives
\begin{enumerate}
\item $(e + \gamma E + \beta H)^\ddag(b). D^{(i)} \bar{\kappa} \equiv 0$
\item $(h - \gamma H)^\ddag(b).  D^{(i)} \bar{\kappa} \equiv 0$
\item $(f + \alpha H + \delta F)^\ddag(b). D^{(i)} \bar{\kappa} \equiv 0$
\end{enumerate}

From (1), 
\begin{eqnarray*}
 D^{(i+1)} \bar{\kappa}_b(e) & = & - (\gamma E + \beta H)^\ddag(b).D^{(i)} \bar{\kappa} \\
& = & (\gamma E + \beta H) \cdot D^{(i)} \bar{\kappa}_b \\
& = & \gamma E \cdot D^{(i)} \bar{\kappa}_b.
\end{eqnarray*}
The last equality above follows from the induction hypothesis.  Then

\begin{eqnarray*}
(H \cdot D^{(i+1)} \bar{\kappa}_b)(e) & = & H \cdot ( D^{(i+1)} \bar{\kappa}_b(e)) - D^{(i+1)} \bar{\kappa}_b([H,e]) \\
& = & H \cdot (\gamma E ) \cdot D^{(i)} \bar{\kappa}_b - D^{(i+1)} \bar{\kappa}_b(e) \\
& = & \gamma ([H,E] + EH) \cdot D^{(i)} \bar{\kappa}_b - \gamma E \cdot (D^{(i)} \bar{\kappa}_b) \\
& = & \gamma E \cdot D^{(i)} \bar{\kappa}_b  - \gamma E \cdot D^{(i)} \bar{\kappa}_b = 0
\end{eqnarray*}
where the last equality again uses the induction hypothesis.  Item (2) gives, by a similar calculation,
$$ D^{(i+1)} \bar{\kappa}_b (h) = - \gamma H \cdot D^{(i)} \bar{\kappa}_b = 0$$
and 
$$ (H \cdot D^{(i+1)} \bar{\kappa}_b)(h) = 0.$$

Finally, (3) gives
$$  D^{(i+1)} \bar{\kappa}_b (f) = \delta F \cdot D^{(i)} \bar{\kappa}_b$$
and again
$$   (H \cdot D^{(i+1)} \bar{\kappa}_b)(f) = 0.$$
We have thus shown vanishing of $H \cdot D^{(i+1)} \bar{\kappa}_b$ on $\RR^{2,1}$.  The remainder of $\lieh$ is obtained by taking linear combinations with $\liep$.  The $H$-invariance of $D^{(i)} \bar{\kappa}$ and $P$-equivariance of $D^{(i+1)} \bar{\kappa}$ give, for $X \in \liep$,
\begin{eqnarray*}
(H \cdot D^{(i+1)}\bar{\kappa_b})(X) & = & H \cdot  (D^{(i+1)}\bar{\kappa}_b(X)) -  D^{(i+1)}\bar{\kappa}_b([H,X]) \\
& = & - H \cdot X \cdot D^{(i)} \bar{\kappa}_b + [H,X] \cdot D^{(i)} \bar{\kappa}_b \\
& = & - X \cdot H \cdot  D^{(i)} \bar{\kappa}_b = 0.
\end{eqnarray*}
The conclusion is $H \cdot D^{(i+1)} \bar{\kappa}_b = 0$, as desired.
\end{proof}

Now if $\bar{\kappa}$ and all its derivatives are constant on $\mathcal{R}$, then $U$ is curvature homogeneous to all orders, and therefore, $U$ is locally homogeneous by Singer's theorem for Cartan connections~\cite{Melnick2, Pecastaing}.

\bigskip

Let us consider now the remaining case where the isotropy at the origin is unipotent.

\subsection{Unipotent isotropy}

\begin{proposition} \label{solab} If  the isotropy at $0 \in S$ is unipotent, then $\lieg$ is isomorphic to  $\mathfrak{sol}(a,b)$, for $b \neq -a$.
\end{proposition}

\begin{proof}
Let $Y \in \lieg$ generate the isotropy at $0$.
There is $b_0 \in \pi^{-1}(0)$ for which $\omega_{b_0} (Y) = E$.  The projection $\overline{\omega_{b_0}(\lieg)}$ of $\omega_{b_0}(\lieg)$ to $\RR^{2,1}$ is 2-dimensional and $E$-invariant, so it must be span$\{ e,h \}$.  Therefore, there is a basis $( X,Y,Z )$ of $\lieg$ such that

$$ \omega_{b_0}(X) = e + \alpha H + \beta F \qquad \mbox{and} \qquad \omega_{b_0} (Z) =  h + \gamma H + \delta F$$
for some $\alpha, \beta, \gamma, \delta \in \RR$.  Because $K_{b_0} (Y, \cdot ) = 0$, equation (\ref{eqn.curv.killing})  gives
$$ \omega_{b_0} [Y,X] = [e + \alpha H + \beta F, E] = \alpha E - \beta H \in \omega_{b_0} (\lieg)$$
so $\beta = 0$ and $[Y,X] = \alpha Y$.  A similar computation gives 
$$ \omega_{b_0}[Y,Z] =  [ h + \gamma H + \delta F,E] = e + \gamma E - \delta H$$ 
so $\delta =  - \alpha$, and $[Y,Z] = X + \gamma Y$.

Infinitesimal invariance of $K$ by $Y$ gives
$$ K_{b_0}([Y,X],Z) + K_{b_0}(X,[Y,Z]) = [ -  E, K_{b_0}(X,Z)].$$
But the left side is 0 because $[Y,X](0) = 0$ and $[Y,Z](0) = X(0)$.  Therefore $E$ commutes with $K_{b_0}(X, Z) \in \liep$, which means
$$ K_{b_0} (X,Z) = rE \qquad \mbox{for some } r \in \RR.$$

Now equation (\ref{eqn.curv.killing})  gives for $X$ and $Z$,
\begin{eqnarray*}
\omega_{b_0} [X,Z] & = & [ h + \gamma H - \alpha F, e + \alpha H]  + rE  \\
& = & \gamma e + \alpha h -  \alpha^2 F + rE.
\end{eqnarray*}

In order that this element belongs to $\omega_{b_0} (\lieg)$, one must have $\alpha = 0$ or $\gamma = 0$.  First consider $\gamma = 0$.   The structure of the algebra $\lieg$ in the basis $( X,Y, Z )$ is
\begin{eqnarray*}
\mbox{ad } Y = 
\left( 
\begin{array}{ccc}
0 &   &  1 \\
\alpha   &    0 &  \\
   &   & 0
   \end{array}
      \right)
   \qquad
   \mbox{ad } X = 
   \left(
   \begin{array}{ccc}
   0 &   &  \\
     &  - \alpha & r  \\
     &   & \alpha
     \end{array}
     \right)
     \qquad
     \mbox{ad } Z = 
     \left(
     \begin{array}{ccc}
    0 & -1 & \\
    -r & 0 &  \\
    - \alpha  &  & 0
      \end{array}
      \right).
      \end{eqnarray*}
This algebra is unimodular, so this case does not arise, by Lemma~\ref{3}.

Next consider $\alpha = 0$.  Then the Lie algebra is
\begin{eqnarray*}
\mbox{ad } Y = 
\left( 
\begin{array}{ccc}
0 &  & 1   \\
   &    0 & \gamma \\
   &   & 0
   \end{array}
      \right)
   \qquad
   \mbox{ad } X = 
   \left(
   \begin{array}{ccc}
 0 &   & \gamma \\
     &  0 & r  \\
     &   & 0
     \end{array}
     \right)
     \qquad
     \mbox{ad } Z = 
     \left(
     \begin{array}{ccc}
    - \gamma & -1 & \\
    -r &      -\gamma  &  \\
        &          & 0
      \end{array}
      \right).
      \end{eqnarray*}
In order that $\lieg$ not be unimodular, $\gamma$ must be nonzero (notice also that for  $\gamma = r = 0$, we would get a Heisenberg algebra).  
We obtain a solvable Lie algebra
$$ \lieg \cong \RR \ltimes_\varphi \RR^2, \qquad \mbox{where} \ 
\varphi = 
\left(
\begin{array}{cc}
-\gamma & -1 \\
-r & - \gamma
\end{array} 
\right).
$$

If $r > 0$, then 
$$ \lieg \cong \mathfrak{sol}(a,b), \qquad \mbox{where } a = - \gamma + \sqrt{r}, \ b = - \gamma - \sqrt{r}.$$
Conversely, $\varphi$ is $\RR$-diagonalizable only if $r > 0$.
\end{proof}

\begin{proposition} (compare Proposition \ref{eigenvalues} ($i$))
\ \ \
\begin{enumerate}[(i)]
\item At points of $S$, there is only one eigenvalue of the Ricci operator. 

\item This triple eigenvalue is positive if and only if the Killing algebra $\mathfrak{sol}(a,b)$  is $\RR$-diagonalizable.
\end{enumerate}
\end{proposition}

\begin{proof}  (i) The invariance of the Ricci endomorphism $\bar{\kappa}_{b_0}$ by $E$ means (see the table in Subsection \ref{subsection.curv.rep}):
$$\bar{\kappa}_{b_0}  \in \mbox{span} \{m_d,  m_{e^2} \}.$$ 
The triple eigenvalue is the coefficient of $m_d$.

(ii) The full curvature $\kappa_{b_0} \in \mathbb{W}$ is $E$-invariant, so it is in the span of the elements of $\mathbb{W}$ corresponding to $m_d$ and $m_{e^2}$. 
Referring to the column labeled $\wedge^2 \RR^{2,1*} \otimes \liep$ in the table reveals that $m_d$ is the only of these two components of $\kappa_{b_0}$ possibly assigning a nonzero value to the input pair $(e,h)$. 
Therefore the parameter $r$ in the proof of Proposition~\ref{solab} coincides with the coefficient of the element corresponding to $m_d$ in $\kappa_{b_0}$ and with half the triple eigenvalue of the Ricci endomorphism at $0$.
\end{proof}

But, by the point (iii) in Proposition~\ref{surface},   we know that the Killing algebra   $\mathfrak{sol}(a,b)$  is $\RR$-diagonalizable. This implies that $r >0$.

On the other hand, recall that in~\cite{CK} Calvaruso and Kowalski classified   Ricci operators for left-invariant  Lorentz metrics $g$ on three-dimensional Lie groups. In particular, they proved (see their Theorems 3.5, 3.6 and  3.7)
 that a Ricci operator  of a left-invariant Lorentz metric on a {\it nonunimodular} three-dimensional Lie group admits a triple eigenvalue $r \neq 0$ if and only if   $g$ is of constant sectional curvature. Since on $U \setminus S$, our Lorentz metric $g$  is locally isomorphic to a left-invariant Lorentz metric on the nonunimodular Lie group $SOL(a,b)$ corresponding  to the Killing algebra, this implies that $g$ is of constant sectional curvature. In particular, $g$ is locally homogeneous.

\bibliography{references}
\bibliographystyle{alpha}

\end{document}